\documentclass[12pt]{amsart}

\usepackage{amsmath, amssymb, amsthm,verbatim,amscd}
\usepackage{bm}
\usepackage[dvipdfmx]{graphicx}
\usepackage{color}
\usepackage{epsfig}
\usepackage[all]{xy}
\usepackage{braket}
\usepackage{enumerate}
\usepackage{mediabb}
\usepackage{overpic}

\theoremstyle{plain}
\newtheorem{thm}{Theorem}[section]

\newtheorem{prob}[thm]{Problem}
\newtheorem{lem}[thm]{Lemma}

\newtheorem{claim}[thm]{Claim}

\newtheorem*{thm-eight}{Theorem 8.1}

\theoremstyle{definition}
\newtheorem{ex}[thm]{Example}
\newtheorem{defn}[thm]{Definition}

\newtheorem*{problem}{Problem 3.6(D)}

\theoremstyle{remark}
\newtheorem{rem}[thm]{Remark}

\newcommand{\N}{\mathbb{N}}
\newcommand{\Z}{\mathbb{Z}}
\newcommand{\R}{\mathbb{R}}

\newcommand{\calL}{\ensuremath{{\mathcal L}}}

\DeclareMathOperator{\nbhd}{Nbhd}

\newcommand{\sign}{\mathop{\mathrm{sign}}\nolimits}

\def\smallcoprod{\raise.3ex\hbox{$\,\scriptstyle\coprod\,$}}

\def\eop#1{\hfill\break\rightline{$\square$\ #1}}

\numberwithin{equation}{section}

\title[Infinitely many knots admitting the same integer surgery]
{Infinitely many knots admitting the same integer surgery and a 4-dimensional extension}

\author[Abe]{Tetsuya Abe} 
\address{Department of Mathematics,
Tokyo Institute of Technology,
2-12-1 Ookayama, Meguro-ku, 
Tokyo 152-8551, Japan}
\email{abe.t.av@m.titech.ac.jp}

\author[Jong]{In Dae Jong}
\address{Department of Mathematics, 
Kinki University, 
3-4-1 Kowakae, Higashiosaka City,
Osaka 577-0818, Japan}
\email{jong@math.kindai.ac.jp}

\author[Luecke]{John Luecke}
\address{University of Texas at Austin}
\email{luecke@math.utexas.edu}

\author[Osoinach]{John Osoinach}
\address{University of Dallas}
\email{josoinach@udallas.edu}


\begin{document}

\maketitle

\begin{abstract}
We prove that for any integer $n$ there exist infinitely many different knots in $S^3$  such that $n$-surgery on those knots yields the same $3$-manifold. 
In particular, when $|n|=1$ homology spheres arise from these surgeries. 
This answers Problem 3.6(D) on the Kirby problem list. 
We construct two families of examples, 
the first by a method of twisting along an annulus 
and the second by a generalization of this procedure. 
The latter family also solves a stronger version of Problem 3.6(D), 
that for any integer $n$, there exist infinitely many mutually distinct knots 
such that 2-handle addition along each with framing $n$ yields the same $4$-manifold. 
\end{abstract}

\section{Introduction}\label{sec:intro} 

Dehn surgery on knots is a long-standing technique for the construction of 3-manifolds.  While well-known theorems of Lickorish~\cite{lickorish} and Wallace~\cite{wallace} state that every orientable 3-manifold can be obtained by Dehn surgery on some link in $S^3,$ this representation is far from unique.  In particular, in the Kirby problem list~\cite{Kirby}, Clark asks the following: 

\begin{problem}
Fix an integer $n$. 
Is there a homology 3-sphere (or any 3-manifold) which can 
be obtained by $n$-surgery on an infinite number of distinct knots?
\end{problem}

\noindent 
In \cite{Osoinach}, the parenthetical version of this question was answered affirmatively  by constructing knots using the method of twisting along an annulus. 
This method was subsequently developed in \cite{teragaito} to construct infinitely many knots yielding a small Seifert-fibered manifold. In \cite{Osoinach}, the surgery slope is $0$, and in \cite{kouno},\cite{teragaito} and \cite{bgl} the surgery slopes are 
multiples of $4$. 

In Section~\ref{sec:LO}, we use the annular twist construction to create, for each integer $n$, an infinite family of distinct knots in $S^3$ such that $n$-surgery on each knot in the collection yields the same manifold (Theorem~\ref{LO-mainthm}). 
When $|n|=1$, the resulting manifold is a homology sphere thereby 
answering affirmatively Problem 3.6(D) above. 
The members of each infinite
family are distinguished by their hyperbolic volume. 
Alternatively, 
at least when $n \neq 0,$ the knots in a family are shown to be different by 
proving that the bridge numbers tend to infinity as the number of twists along the annulus increases. 

In \cite{AJOT}, 
a $4$-dimensional extension of Problem 3.6(D) was proposed as follows: 

\begin{prob}\label{prob:Kirby}
Let $n$ be an integer. 
Find infinitely many mutually distinct knots $K_1, K_2, \dots$ 
such that $X_{K_i}(n) \approx X_{K_j}(n) $ for each $i, j \in \N$. 
\end{prob}
\noindent 
Here $X_{K}(n)$ denotes the smooth $4$-manifold obtained from 
the $4$-ball $B^4$ by attaching a $2$-handle along $K$ with framing $n$, 
and the symbol $\approx$ stands for a diffeomorphism. 

In Section~\ref{sec:AJ}, 
we generalize the annulus twist method in a somewhat surprising way 
to produce a different family of knots answering Problem 3.6(D). 
Furthermore, this family solves Problem~\ref{prob:Kirby} affirmatively as follows. 

\begin{thm}\label{thm:main}
For every $n \in \Z$, 
there exist distinct knots $J_0, J_1, J_2, \dots$ such that
\[ X_{J_0}(n) \approx X_{J_1}(n) \approx X_{J_2}(n) \approx \cdots \, . \]
\end{thm}

The knots $J_0$ and $J_1$ 
in Theorem~\ref{thm:main} (for $n>0$) are depicted in Figure~\ref{fig:8_20-1}, 
where the rectangle labelled $n$ stands for $n$ right-handed full twists.
Note that $J_0$ is the knot $8_{20}$ in Rolfsen's table~\cite{Rolfsen}. 
The members of each infinite family 
are distinguished by their Alexander polynomials  when $n \ne 0$. 
When $n = 0$, they are distinguished by hyperbolic volume (see \cite{AJOT}). 

\begin{figure}[!htb]
\centering
\begin{overpic}[width=0.65\textwidth]{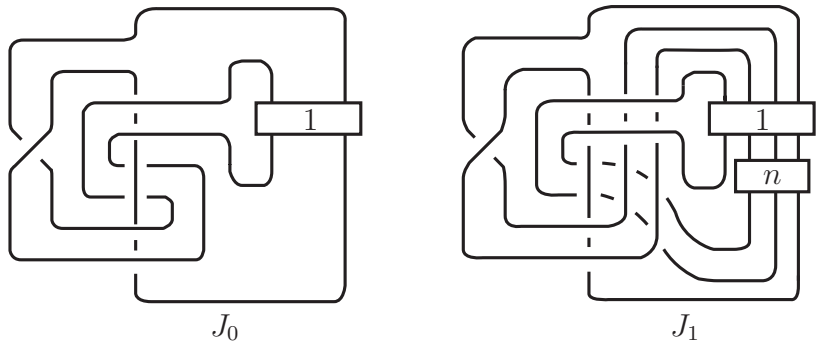}
\put(25,-4){$J_0$}
\put(82,-4){$J_1$}
\put(36.5,21.5){$1$}
\put(92.5,21.5){$1$}
\put(93.5,14.8){$n$}
\end{overpic}
\caption{The knots $J_0$ and $J_1$ such that $X_{J_0}(n) \approx X_{J_1}(n)$. }
\label{fig:8_20-1}
\end{figure}

\subsection*{Acknowledgments}
The first and second authors would like to express their gratitude to 
Yuichi Yamada and other participants of 
the handle seminar organized by Motoo Tange. 
Section~\ref{sec:AJ} would not have arisen 
without Yamada's interest in annulus twists. 
The third and fourth authors would like to thank Kyle Larson for very helpful conversations, and Neil Hoffman for his help with HIKMOT. 
The authors also thank the referees for careful reading of our draft 
and helpful suggestions. 
The first author was supported by JSPS KAKENHI Grant Number 13J05998.

\section{First family of knots}\label{sec:LO}

The Dehn surgeries on a knot, $K$, 
in the $3$-sphere are parameterized by their surgery slopes. These surgery slopes are described by $p/q \in \mathbb{Q} \cup \{\infty\}$, meaning that the slope is a curve that runs $p$ times meridionally and $q$ times longitudinally (using the preferred longitude)
along the boundary of the exterior of $K$. We write $M_K(p/q)$ for the $p/q$ 
Dehn surgery on $K$. In this notation, an $n$-surgery on $K$ refers to the
integer surgery $M_K(n/1)=M_K(n)$. 

\begin{defn}\label{defL}
Let $\calL = k \cup l_1 \cup l_2 \cup l_3$ be the link pictured in Figure~\ref{fig:Example}.
Let $\calL(\alpha, \beta, \delta, \gamma)$ be the corresponding Dehn surgery on $\calL$.
Here the surgery slopes $\alpha,\beta,\delta, \gamma$ will be either in $\mathbb{Q} \cup \{ \infty \}$,
using the meridian-longitude coordinates on the boundary of a knot in $S^3$ (with a right-handed orientation
on $S^3$), or an asterisk, meaning
that no surgery is done on that component and the component is seen as a knot in the
surgered manifold. (We use the notation $\calL(\alpha, \beta, \delta, \gamma)$ rather than
$M_\calL(\alpha, \beta, \delta, \gamma)$, because, when there are asterisks among the arguments, 
this denotes a link in the surgered manifold.) 
\end{defn}

\begin{figure}[!htb]
\centering
\begin{overpic}[width=.5\textwidth]{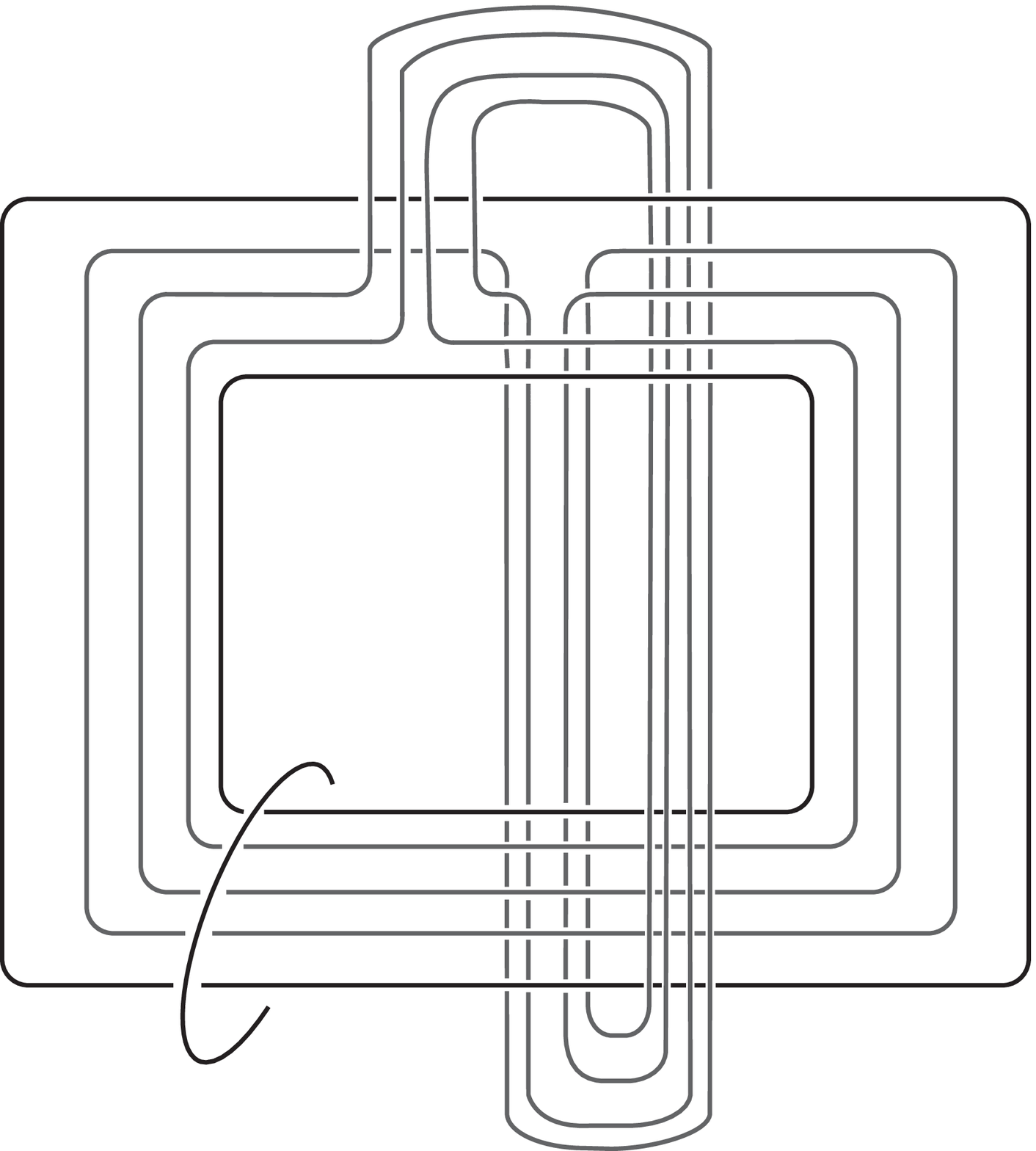}
\put(29,95){$k$}
\put(1,10.5){$l_1$}
\put(22,7){$l_3$}
\put(20.5,50){$l_2$}
\end{overpic}
\caption{The link $\calL$}
\label{fig:Example}
\end{figure}

The main result of this section, giving our first family of knots by surgery on the link
$\calL$, is the following theorem.

\begin{thm}\label{LO-mainthm}
For integers $m$, $n$, $k_n^m=\calL(*,-1/m,1/m,-1/n)$ is a knot in $S^3$. Furthermore $M_{k_n^{m_1}}(n)$
is homeomorphic to $M_{k_n^{m_2}}(n)$ for any integers $m_1,m_2$.
\begin{enumerate}
\item 
For a fixed $n \neq 0$,
the bridge number of $k_n^m$ tends to infinity as $m$ tends to infinity.
\item 
For any integer $n$, 
there is a $C_n>0$ such that if $m_2 > m_1 > C_n$, 
then $k_n^{m_1}$ and $k_n^{m_2}$ are hyperbolic knots 
with the hyperbolic volume of $k_n^{m_2}$ larger than that of $k_n^{m_1}$.
\end{enumerate}
In particular, for each integer $n$ there are infinitely many different knots in the 
family $\{k_n^m\}$.
\end{thm}

\noindent{\it Proof.}

We first show that for any integer $n$, the $n$-surgery on each $k_n^m$ yields the same
manifold for each $m$.
Figure~\ref{fig:Q} shows that the knot $k$ is a non-separating, orientation-preserving curve on a twice-punctured Klein bottle,
$Q$, cobounded by $l_1$ and $l_2$ and in the complement of $l_3$.

\begin{figure}[!htb]
\centering
\begin{overpic}[width=.5\textwidth]{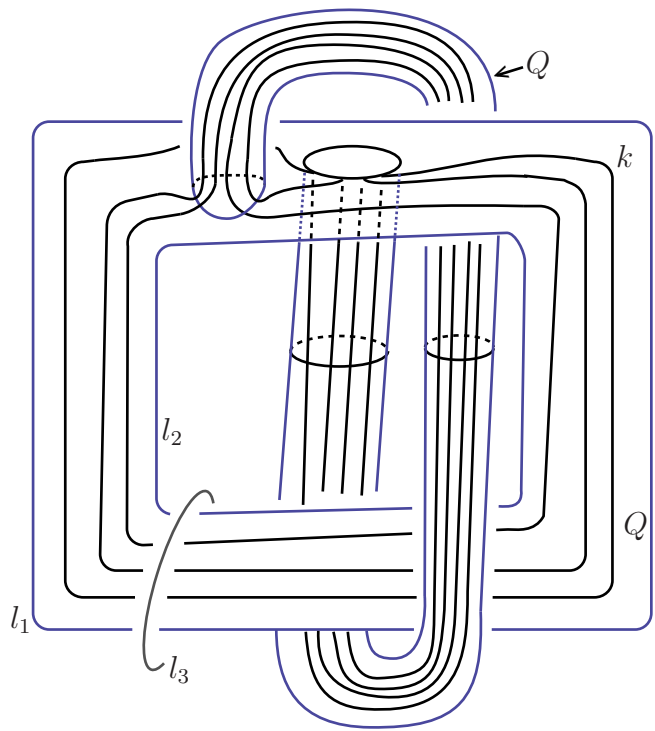}
\put(-3,14){$l_1$}
\put(18,40){$l_2$}
\put(19,7){$l_3$}
\put(68.5,91){$Q$}
\put(82,27){$Q$}
\put(81,78){$k$}
\end{overpic}
\caption{The $2$-punctured Klein bottle $Q$ containing $k$}
\label{fig:Q}
\end{figure}

Thus $Q - \nbhd(\calL)$ is a $4$-punctured sphere, $P$, 
properly embedded in the exterior, $E_{\calL}$, of $\calL$ in $S^3$. 
The boundary of $P$ has one component on 
each of $\partial \nbhd(l_1)$ and $\partial \nbhd(l_2)$ of slope $0/1$ and 
two components on $\partial \nbhd(k)$ of slope $0/1$. 
To check that the slope of $P$ on $\partial \nbhd(k)$ is $0/1$, 
one can verify that the linking number of such a boundary component is zero 
with respect to $k$. 
To do this, it is convenient to use $P$ as in the proof of Claim~\ref{clm:slope},  below, when $m=n=0$. 

Let $\widehat{P}$ be the properly embedded annulus in the exterior of
$\calL(0/1,*,*,*)$ obtained by capping off the two components of $P$ along $\partial \nbhd(k)$.
Dehn twisting this exterior along the annulus $\widehat{P}$ $m$ times 
(see Remark~\ref{twisting}), induces a homeomorphism
of the $3$-manifolds $\calL(0/1,-1/m,1/m,-1/n)$ and $\calL(0/1,-1/0,1/0,-1/n)$ for each $m,n$.

\begin{rem}\label{twisting}
Let $A$ be an annulus embedded in a $3$-manifold $M$ 
with $\partial A$ the link $L_1 \cup L_2$ in $M$. 
Let $A'=A \cap (M - \nbhd(L_1 \cup L_2))$. 
Fix an orientation on $M$. 
Pick an orientation on $A$. 
This induces an orientation on $L_i$ and its meridian $\mu_i$.
Let $A \times [0,1]$ be a product neighborhood of $A$ in $M$ 
so that the corresponding interval orientation on $A' \times [0,1]$ 
corresponds to the meridian orientation of $L_1$. 
Pick coordinates $A=e^{2 \pi i \theta} \times [0,1]$, with $\theta \in [0,1]$, 
so that $e^{2 \pi i \theta} \times \{ 0 \}$, 
$\theta \in [0,1]$, is the oriented $L_1$. 
Define the homeomorphism $f_m \colon  A \times [0,1] \to A \times [0,1]$ by 
$(e^{2 \pi i \theta},s,t) \mapsto (e^{2 \pi i(\theta + mt)},s,t)$. 
Note that $f_m$ restricted to $A \times \{0,1\}$ is the identity. 
Let $A$ be as above and $K$ be a knot in $M$ 
which intersects $A \times [0,1]$ in $[0,1]$ fibers. 
Let $K^m$ be the knot in $M$ obtained by applying
$f_m$ to $K \cap (A \times [0,1])$ (and the identity on $K$ outside this region). 
We say that $K^m$ is obtained from $K$ by {\em twisting along $A$ ($m$ times)}, 
or that $K^m$ is obtained from $K$ 
by applying an {\em $m$-fold annulus twist along $A$}. 
In particular, we say that $K^1$ is the result of applying to $K$ an {\em annulus twist along $A$}. 
Note that the sign of $m$ above depends only on the orientation of $M$ 
and on the labeling, $L_1$ and $L_2$, of $\partial A$. 
The above agrees with the notion of an annulus twist along $A$ in \cite[Section 2]{AT}, 
where $M = S^3$ with a right-handed orientation, $A$ is a planar annulus, $L_1$ is the outside boundary of $A$, and $L_2$ is the inside boundary of $A$. 
The manifolds $M$ with which we are working are $S^3$ or Dehn surgeries on $S^3$.  Our convention, is to take the right-handed orientation of $S^3$ and the induced orientation on these Dehn surgeries. 
Furthermore, note that $f_m$ induces a homeomorphism
$h_m \colon M-\nbhd(L_1 \cup L_2) \to M-\nbhd(L_1 \cup L_2)$ by applying $f_m$
in $A' \times [0,1]$ along with the identity outside this neighborhood. 
We refer to this homeomorphism $h_m$ of $M-\nbhd(L_1 \cup L_2)$ 
as {\em Dehn-twisting along $A'$ ($m$ times)}. 
In this case, $A'$ is properly embedded. 
\end{rem}

\begin{figure}[!htb]
\centering
\begin{overpic}[width=.3\textwidth]{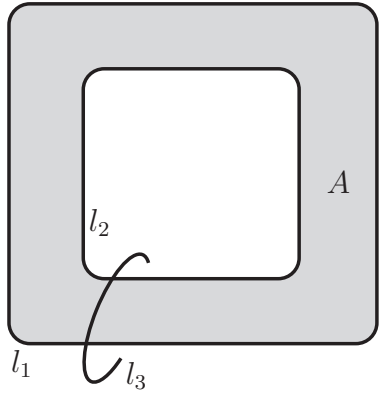}
\put(1,3){$l_1$}
\put(21.3,40){$l_2$}
\put(31,0){$l_3$}
\put(83.5,50){$A$}
\end{overpic}
\caption{The annulus $A$ bounded by $l_1 \cup l_2$}
\label{fig:A}
\end{figure}

\bigskip

Figure~\ref{fig:A} shows an annulus $A$ cobounded by $l_1$ and $l_2$ in the complement of $l_3$ (which can be taken to
intersect $k$ algebraically zero and geometrically four times and which induces the framing $0/1$ on each of $l_1$ and $l_2$), which becomes an annulus $A_n$ cobounded by $l_1$ and $l_2$ after $-1/n$ surgery on $l_3$. 
Dehn-twisting the exterior of $l_1 \cup l_2$ in $\calL(*,*,*, -1/n)$ 
along $A_n$ $(-m)$ times 
(really the restriction of $A_n$ to this exterior, Remark~\ref{twisting}) 
induces an orientation-preserving homeomorphism of the manifold 
$\calL(1/0,-1/0,1/0,-1/n)=S^3$ to the manifold $\calL(1/0, -1/m, 1/m, -1/n)$. 
The inverse of this 
homeomorphism identifies $k_n^m$ as a knot in $S^3$ 
obtained from $k_n^0$ by twisting along $A_n$ $m$ times (see Remark~\ref{twisting}). 

The following claim finishes the argument that the $n$-surgeries on $k_n^m$ are the same manifold.

\begin{claim}\label{clm:slope}
For each $m$, $n$, $\calL(0/1,-1/m,1/m,-1/n)=M_{k_n^m}(n)$.
\end{claim}

\noindent {\it Proof of Claim~\ref{clm:slope}:} 
$\calL(0/1,-1/m,1/m,-1/n)$ is clearly a surgery on $k_n^m$. 
Our goal is to identify the slope of this surgery, $\alpha(m,n)$, in terms of the coordinates on $k_n^m$ as a knot in $S^3$.
Let $P_n$ be the $4$-punctured sphere $P$ after $-1/n$ surgery on $l_3$. Then $\alpha(0,n)$ is the slope of $P_n$ on $k_n^0$.

Twisting $k_n^0$ along $A_n$ induces a homeomorphism of the exterior 
of $l_1 \cup l_2 \cup k_n^0$ in $S^3$ to the exterior of $l_1 \cup l_2 \cup k_n^m$ and
consequently takes $P_n$ to a $4$-punctured sphere $P_n^m$ in the exterior of $l_1 \cup
l_2 \cup k_n^m$. The slope $\alpha(m,n)$ is the slope of $P_n^m$ on $k_n^m$. We may use
$P_n^m$ to compute the linking number of the slope $\alpha(m,n)$ with $k_n^m$ and consequently the
coordinates of the slope. Orient $k_n^m$ and take the orientation on $P_n^m$ that induces an
orientation on $\partial P_n^m \cap \nbhd(k_n^m)$ that agrees with that on $k_n^m$. Then twice
the linking number of $\alpha(m,n)$ with the oriented $k_n^m$ in $S^3$ is the negative of the linking number between the oriented $k_n^m$ and $l_1 \cup l_2$, given the orientation induced by $P_n^m$ on $l_1 \cup l_2$. By considering $k_n^m$ as twisting $k_n^0$ along $A_n$ away from $l_1 \cup l_2$,
one sees that this latter linking number is $-2n$ (one may verify that in the $1/0$ surgery on $l_3$, this linking number is zero, then observe how the linking number changes under $-1/n$ surgery). 
Thus $\alpha(m,n)$ is the slope $n/1$ as desired.
\eop{(Claim~\ref{clm:slope})}

\begin{claim}\label{clm:noannulus}
Let $E_n$ be the exterior of $\calL(*,*,*,-1/n)$ and $T_1,T_2$ be the components of
$\partial E_n$ coming from $\nbhd(l_1),\nbhd(l_2)$, respectively.
For each integer $n \neq -2$, the interior of $E_n$ is hyperbolic.
For every integer $n$ (including $-2$), there is no essential annulus properly 
embedded in 
$E_n$ with
one boundary component on $T_1$ and the other on $T_2$.
\end{claim}

\noindent {\it Proof of Claim~\ref{clm:noannulus}:}
SnapPy \cite{snappy} shows that $\calL$ is hyperbolic. The program HIKMOT 
\cite{hikmot} certifies this calculation.
The sequence of isotopies Figure~\ref{fig:2surgery}(a)-(c) shows that  $l_1$ in  $\calL(*,*,*,1/2)$ is a $(2,-1)$-cable on the knot $l_1'$ pictured
in Figure~\ref{fig:2surgery}(d) (the 3-manifold $H$ in Figure~\ref{fig:2surgery} is a
neighborhood of the punctured Klein bottle $Q$ and $l_1$ is pushed
off $H$). Because the linking number of $l_1'$ with $k$ is one, the exterior of $k \cup l_1 \cup l_2$  in
$\calL(*,*,*,1/2)$ is toroidal. It follows from \cite{gordon} 
and \cite{gordon-wu} that the interior of $E_n$ is
hyperbolic as long as $|n+2|>3$.

For $n \in \{1,0,-1,-3,-4,-5 \}$, SnapPy shows that $E_n$ is hyperbolic 
and HIKMOT certifies this calculation. 
Thus the interior of $E_n$ is hyperbolic, and in particular $E_n$ is anannular,
as long as $n \neq -2$.

We must still show that $E_{-2}$ is anannular. As mentioned above,
Figure~\ref{fig:2surgery}(d) shows that $E_{-2}$ is the union, along a
torus $T$, of the exterior of a $(2,-1)$-cable of the core
of a solid torus and
the exterior, $E_{-2}'$, of $l_1' \cup l_2 \cup k$ after $1/2$ surgery on $l_3$. 
SnapPy shows $E_{-2}'$ is hyperbolic and HIKMOT certifies this. 
Now assume that there
were an essential annulus in $E_{-2}$ between $T_1$ and $T_2$, 
and consider its intersection with the incompressible torus $T$. 
We may surger away any closed curves of intersection which are trivial on $T$. 
Then an outermost component of intersection with $E_{-2}'$ 
will give rise to an essential annulus or disk properly embedded in $E_{-2}'$,  contradicting the hyperbolicity of $E_{-2}'$. 
\eop{(Claim~\ref{clm:noannulus})}

\begin{figure}[!htb]
\centering
\begin{overpic}[width=.55\textwidth]{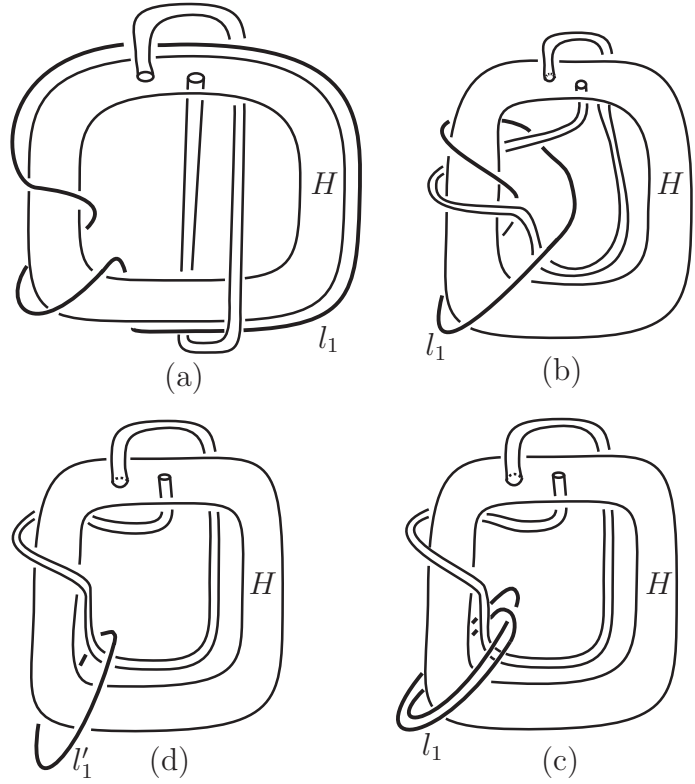}
\put(20,50){(a)}
\put(69,50.8){(b)}
\put(69,0){(c)}
\put(18,0){(d)}
\put(39,75){$H$}
\put(84,75){$H$}
\put(82.5,23){$H$}
\put(31,23){$H$}
\put(40,55){$l_1$}
\put(54,54){$l_1$}
\put(53.5,2){$l_1$}
\put(8,0){$l_1'$}
\end{overpic}
\caption{ $l_1$ in $\calL(*,*,*,1/2)$ is a $(2,-1)$-cable on $l_1'$}
\label{fig:2surgery}
\end{figure}

\bigskip 

We first verify $(1)$ of Theorem~\ref{LO-mainthm}.
As before, let $A_n$ be the annulus from Figure~\ref{fig:A} cobounded by $l_1$ and $l_2$  and after $-1/n$ surgery on $l_3$. The
knot $k_n^m$ is obtained by twisting $k$  along $A_n$ ($m$ times) in the copy of $S^3$ obtained by $-1/n$ surgery
on $l_3$.  As the linking number of $l_1$ and $l_2$ in this copy of $S^3$ is $n$, $l_1 \cup l_2$ is not
the trivial link. Then Claim~\ref{clm:noannulus} along with Corollary 1.4 of \cite{bgl} shows that for $n \neq 0$ the (genus $0$) bridge number of the knots $k_n^m$ in $S^3$ goes to infinity  as $m$ goes to infinity
(as the linking number of $l_1$ and $l_2$ is non-zero, Lemma 2.4 of \cite{bgl} shows there is a catching surface for the
pair $(A_n, k)$).
Note that since $A_0$ lies on a Heegaard sphere for $S^3$, 
the bridge numbers of $\{k_0^m\}$ will be bounded.

We now verify $(2)$ of Theorem~\ref{LO-mainthm}. By Claim~\ref{clm:noannulus},
the interior of $E_n$ is hyperbolic whenever $n \neq -2$. 
Thurston's Dehn Surgery Theorem and Theorem 1A of \cite{neumann} 
shows that there is an $C_n>0$
such that for $m >C_n$, $k_n^m$ is hyperbolic and its volume increases monotonically with
$m$. When $n = -2$, recall from the proof of Claim~\ref{clm:noannulus} that
Figure~\ref{fig:2surgery}(d) shows that $E_{-2}$ is the union, along a
torus $T$, of the exterior of a $(2,-1)$-cable of the core of a solid torus, 
and the exterior, $E_{-2}'$, of $l_1' \cup l_2 \cup k$ after $1/2$ surgery on $l_3$.
That is, identify $\calL(*,*,*,1/2)$ as a link in $S^3$ by putting two full 
left-handed twists along the linking circle $l_3$. Then $\calL(*, -1/m,1/m,1/2)$
corresponds to $(-1-2m)/m$ surgery on $l_1$ and $(1-2m)/m$ surgery on $l_2$.
The Seifert fiber on $l_1$ as a $(2,-1)$-cabling on $l_1'$ is $-2/1$. As the 
surgery slope intersects this Seifert fiber slope once, this surgery on
$l_1$ corresponds to doing a $(-1-2m)/4m$ surgery  on $l_1'$ 
(see Corollary 7.3 of \cite{gordon2}). As noted above,
HIKMOT verifies $k \cup l_1' \cup l_2$ to be hyperbolic.
Thus, an application of Theorem 1A of \cite{neumann} to the exterior $E_{-2}$
of this link, shows there is a $C_{-2}$ such that for $m > C_{-2}$, $k_{-2}^m$
is hyperbolic and its volume increases monotonically with $m$.

Since hyperbolic volume and bridge number are knot invariants, either $(1)$
 (when $n \neq 0$)
or $(2)$ shows that for an integer $n$ the family $\{k_n^m\}$ is infinite.
\eop{(Theorem~\ref{LO-mainthm})}

\begin{rem} SnapPy shows the homology spheres that arise 
in the above construction ($|n|=1$) to be hyperbolic manifolds with 
$volume(M_{k_{-1}^0}(-1))=3.400436870$ and $volume(M_{k_{1}^0}(1))=5.7167678901$. 
SnapPy shows the manifold corresponding
to $n=-2$ to be hyperbolic with $volume(M_{k_{-2}^0}(-2))=3.110698158$.
These calculations are not verified by HIKMOT.
\end{rem}

The next section shows, for each $n$, other infinite families of knots that admit 
the same $n$-surgery. 
We show that in fact the $4$-manifolds obtained by attaching a $2$-handle
to the $4$-ball along each of the knots in one of these families are diffeomorphic. 
We do not know if the same holds for the above family $\{k_n^m\}$.

\bigskip

\noindent {\bf Question} {\it Let $n$ be an integer. Are the $4$-manifolds $X_{k_n^{i}}(n)$ and
$X_{k_n^{j}}(n)$ diffeomorphic?}

\section{Second family of knots}\label{sec:AJ}

We generalize the annulus twist and provide a framework for creating knots 
yielding the same $4$-manifold. 
Problem~\ref{prob:Kirby} is solved by applying the framework to the knot $8_{20}$. 

This section is organized as follows: 
In subsection~\ref{ssec:construction}, 
we recall the definition of an annulus presentation of a knot 
and introduce the notion of a ``simple'' annulus presentation. 
We define a new operation $(*n)$ on an annulus presentation, 
which is a generalization of an annulus twist. 
For a knot $K$ with an annulus presentation and an integer $n$, 
we construct a knot $K'$ (with an annulus presentation) such that 
$M_{K}(n) \approx M_{K'}(n)$ by using the operation $(*n)$ (Theorem~\ref{thm:diffeo3}). 
In subsection~\ref{ssec:extension}, 
for a knot $K$ with a simple annulus presentation and any integer $n$, 
we construct a knot $K'$ (with a simple annulus presentation) such that 
$X_{K}(n) \approx X_{K'}(n)$ by using the operation $(*n)$ (Theorem~\ref{thm:diffeo4}). 
Note that the two knots $K$ and $K'$ are possibly the same. 
In subsection~\ref{ssec:proof}, 
we introduce the notion of a ``good'' annulus presentation, and 
show that, for a given knot with a good annulus presentation, 
the infinitely many knots constructed by using the operation $(*n)$ 
have mutually distinct Alexander polynomials when $n \ne 0$ (Theorem~\ref{thm:main2}). 
This yields Theorem~\ref{thm:main} as an immediate corollary.

\subsection{Construction of knots}\label{ssec:construction}

\subsubsection{Annulus presentation}\label{sssec:AP} 
We recall the definition of an annulus presentation\footnote{In \cite{AJOT},
it was called a \emph{band presentation}.} of a knot from \cite{AJOT, AT}. 
Let $A \subset \R^2 \cup \{ \infty \} \subset S^3$ be a trivially embedded annulus
with an $\varepsilon$-framed unknot $c$ in $S^3$ 
as shown in the left side of Figure~\ref{fig:Def-AP}, 
where $\varepsilon = \pm 1$. 
Take an embedding of a band $b \colon I \times I \to S^3$ such that 
\begin{itemize}
\item $b(I \times I) \cap \partial A = b(\partial I \times I)$, 
\item $b(I \times I) \cap \text{int} A$ consists of ribbon singularities, and
\item $b(I \times I)  \cap c= \emptyset$,
\end{itemize}
where $I = [0,1]$. 
Throughout this paper, we assume that $A \cup b(I \times I)$ is orientable. 
This assumption implies that the induced framing is zero (see \cite{AJOT}). 
Unless otherwise stated, we also assume for simplicity that $\varepsilon=-1$. 
If a knot $K$ in $S^3$ is isotopic to the knot 
$\left( \partial A \setminus b(\partial I \times I)\right) \cup b( I \times \partial I)$ 
in $M_{c}(-1) \approx S^3$, 
then we say that $K$ admits an {\em annulus presentation} $(A,b,c)$.
It is easy to see that a knot admitting an annulus presentation 
is obtained from the Hopf link by a single band surgery (see~\cite{AJOT}). 
A typical example of a knot admitting an annulus presentation is given 
in Figure~\ref{fig:Def-AP}. 

\begin{figure}[!htb]
\centering
\begin{overpic}[width=1.0\textwidth]{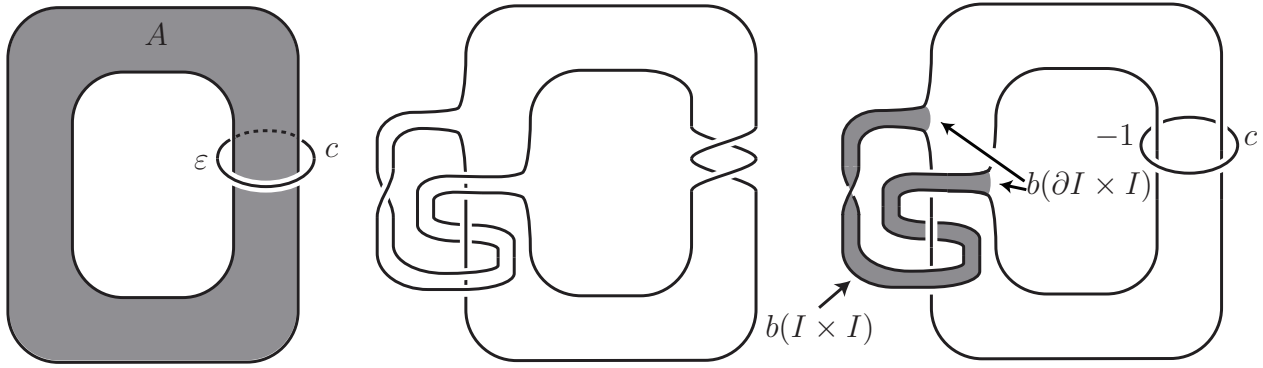}
\put(11.5,26){$A$}
\put(15.5,16){$\varepsilon$}
\put(26,17){$c$}
\put(61.5,2.5){$b(I \times I)$}
\put(82.5,14){$b(\partial I \times I)$}
\put(88,18){$-1$}
\put(100,18){$c$}
\end{overpic}
\caption{The knot depicted in the center admits an annulus presentation as in the right side.}
\label{fig:Def-AP}
\end{figure}

For an annulus presentation $(A,b,c)$, 
$\left( \R^2 \cup \{\infty\} \right) \setminus \textrm{int}A$ consists of 
two disks $D$ and $D'$, see Figure~\ref{fig:simple}. 
Assume that $\infty \in D'$. 

\begin{defn}\label{def:simple} 
An annulus presentation $(A,b,c)$ is called \textit{simple} if 
$b(I \times I) \cap \textrm{int} D = \emptyset$. 
\end{defn}

For example, in Figure~\ref{fig:simple}, 
the annulus presentation depicted in the center is simple, 
and the right one is not.

\begin{figure}[!htb]
\centering
\begin{overpic}[width=1.0\textwidth]{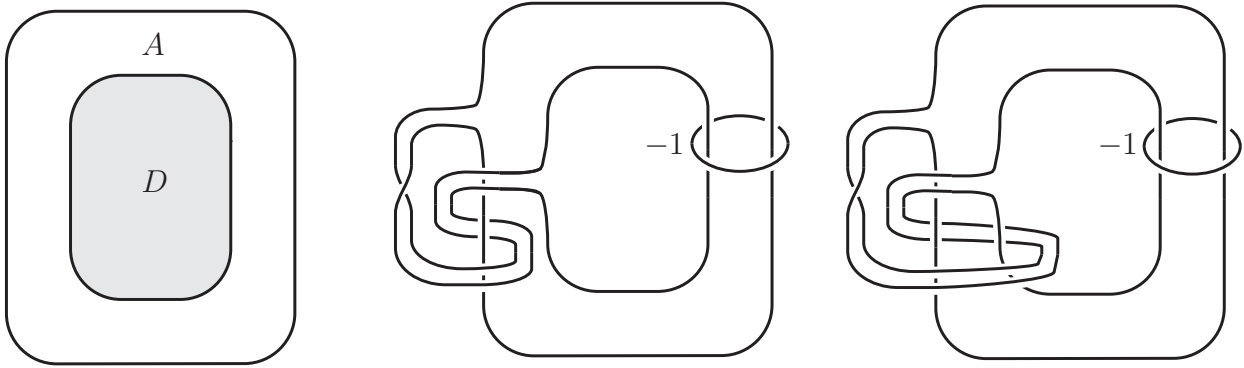}
\put(11,14){$D$}
\put(11,25){$A$}
\put(51.5,17){$-1$}
\put(88,17){$-1$}
\end{overpic}
\caption{The position of $D$, a simple annulus presentation and a non-simple annulus presentation.}
\label{fig:simple}
\end{figure}

Let $(A,b,c)$ be an annulus presentation of a knot. 
In a situation where it is inessential how the band $b(I \times I)$ is embedded, 
we often indicate $(A,b,c)$ in an abbreviated form as in Figure~\ref{fig:AP-abb}.

\begin{figure}[!htb]
\centering
\begin{overpic}[width=.25\textwidth]{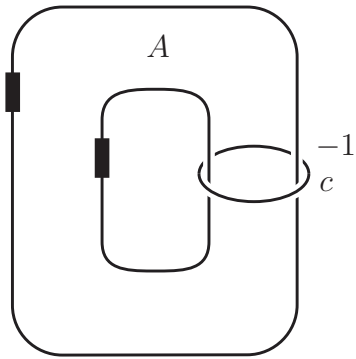}
\put(40,85){$A$}
\put(88,57){$-1$}
\put(89,47){$c$}
\end{overpic}
\caption{
Thick arcs stand for $b (\partial I \times I)$. }
\label{fig:AP-abb}
\end{figure}

\subsubsection{Operations}\label{sssec:operation}

To construct knots yielding the same $4$-manifold by a $2$-handle attaching, 
we define operations on an annulus presentation. 

\begin{defn}
Let $(A, b, c)$ be an annulus presentation, and $n$ an integer. 
\begin{itemize}
\item 
{\em The operation $(A)$} is to apply an annulus twist
along the annulus $A$. 
\item 
{\em The operation $(T_n)$} is defined as follows: 
\begin{enumerate}
\item 
Adding the $(-1/n)$-framed unknot as in Figure~\ref{fig:Tn}, and 
\item 
(after isotopy) blowing down along the $(-1/n)$-framed unknot. 
\end{enumerate}
\item 
{\em The operation $(*n)$} is the composition of $(A)$ and $(T_n)$. 
\end{itemize}
\end{defn}

In the operation $(T_n)$, 
the added $(-1/n)$-framed unknot is lying on the neighborhood of 
$c$ and $\partial A$, and does not intersect $b(I \times I)$. 
The intersection of $A$ and the added unknot is just one point. 

The operation $(*n)$ is a generalization of an annulus twist, in particular, $(*0) = (A)$.

\begin{figure}[!htb]
\centering
\begin{overpic}[width=.6\textwidth]{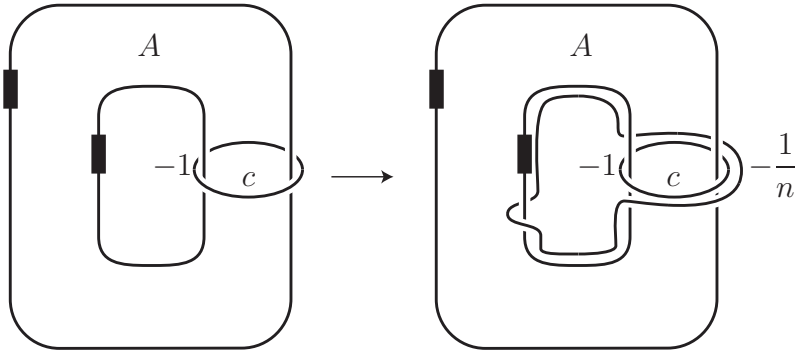}
\put(20,24){$-1$}
\put(77,24){$-1$}
\put(100,24){$-\dfrac1n$}
\put(18,40){$A$}
\put(76,40){$A$}
\put(32,22.5){$c$}
\put(89,22.5){$c$}
\end{overpic}
\caption{Add the $(-1/n)$-framed unknot in the operation $(T_n)$.}
\label{fig:Tn}
\end{figure}

\subsubsection{Construction}\label{sssec:Construction}

For a given knot $K$ with an annulus presentation, 
we can obtain a new knot $K'$ with a new annulus presentation 
by applying the operation $(*n)$. 
By abuse of notation, 
we call  $K'$ {\em the knot obtained from $K$ by the operation $(*n)$}. 
Here we give examples. 

\begin{ex}\label{ex:1}
Let $J_{0}$ be the knot with the simple annulus presentation of Figure~\ref{fig:*n}. 
Let $J_1$ be the knot obtained from $J_{0}$ by the operation $(*{n})$. 
Then $J_1$ is as in Figure~\ref{fig:*n}. 
\end{ex}
\begin{figure}[!htb]
\centering
\begin{overpic}[width=1.0\textwidth]{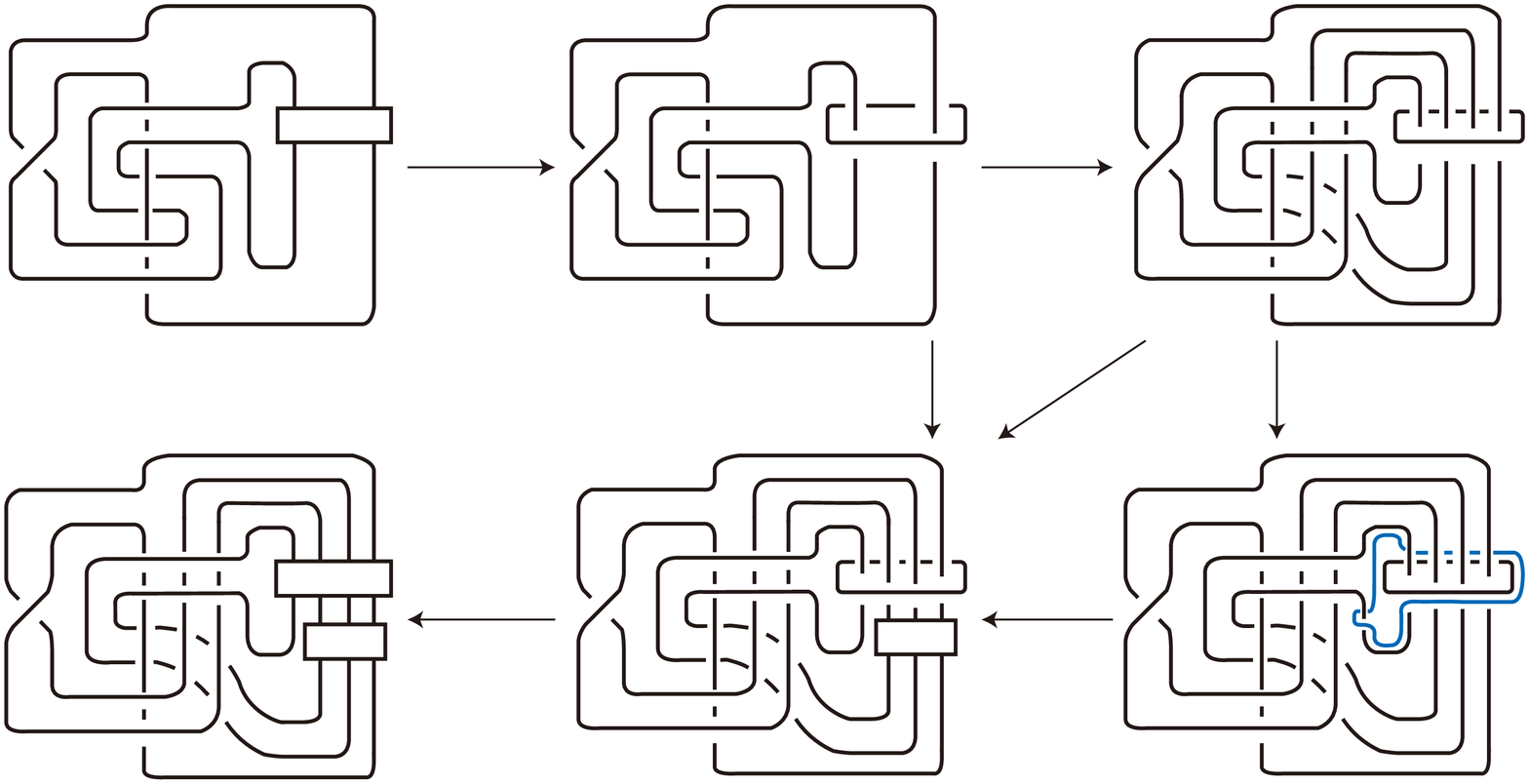}
\put(21.2,42.2){$1$}
\put(37,27.5){annulus presentation}
\put(37,-2){annulus presentation}
\put(26.5,41){blow up}
\put(62,44.5){$-1$}
\put(99,44.3){$-1$}
\put(62.3,14.5){$-1$}
\put(59.4,8.7){$n$}
\put(21.5,12.35){$1$}
\put(21.8,8.4){$n$}
\put(67,41){$(A)$}
\put(66,14.5){blow}
\put(66,12){down}
\put(29,14.5){blow}
\put(29,12){down}
\put(69.5,23){$(T_n)$}
\put(61.5,25.5){$(*n)$}
\put(11.5,27.5){$J_0$}
\put(11.5,-2.5){$J_1$}
\put(98.5,10){$\textcolor[cmyk]{1,0.5,0,0}{-\dfrac1n}$}
\end{overpic}
\vskip .3cm
\caption{By the operation $(*n)$, 
the knot $J_0$ with the annulus presentation is deformed into 
the knot $J_1$ with the annulus presentation.} 
\label{fig:*n}
\end{figure}

\begin{rem}
Let $K$ be a knot with an annulus presentation $(A, b, c)$, 
and $K'$ the knot obtained from $K$ by $(*n)$. 
If $(A, b, c)$ is simple, 
then the resulting annulus presentation of $K'$ is also simple. 
\end{rem}

\begin{ex}\label{ex:2}
For the knot $J_1$ in Example~\ref{ex:1} with $n = 1$, 
let $J_2$ be the knot obtained from $J_{1}$ by applying the operation $(*{1})$. 
Then $J_2$ is as in Figure~\ref{fig:ExK2}.
\end{ex}

\begin{figure}[!htb]
\centering
\begin{overpic}[width=\textwidth]{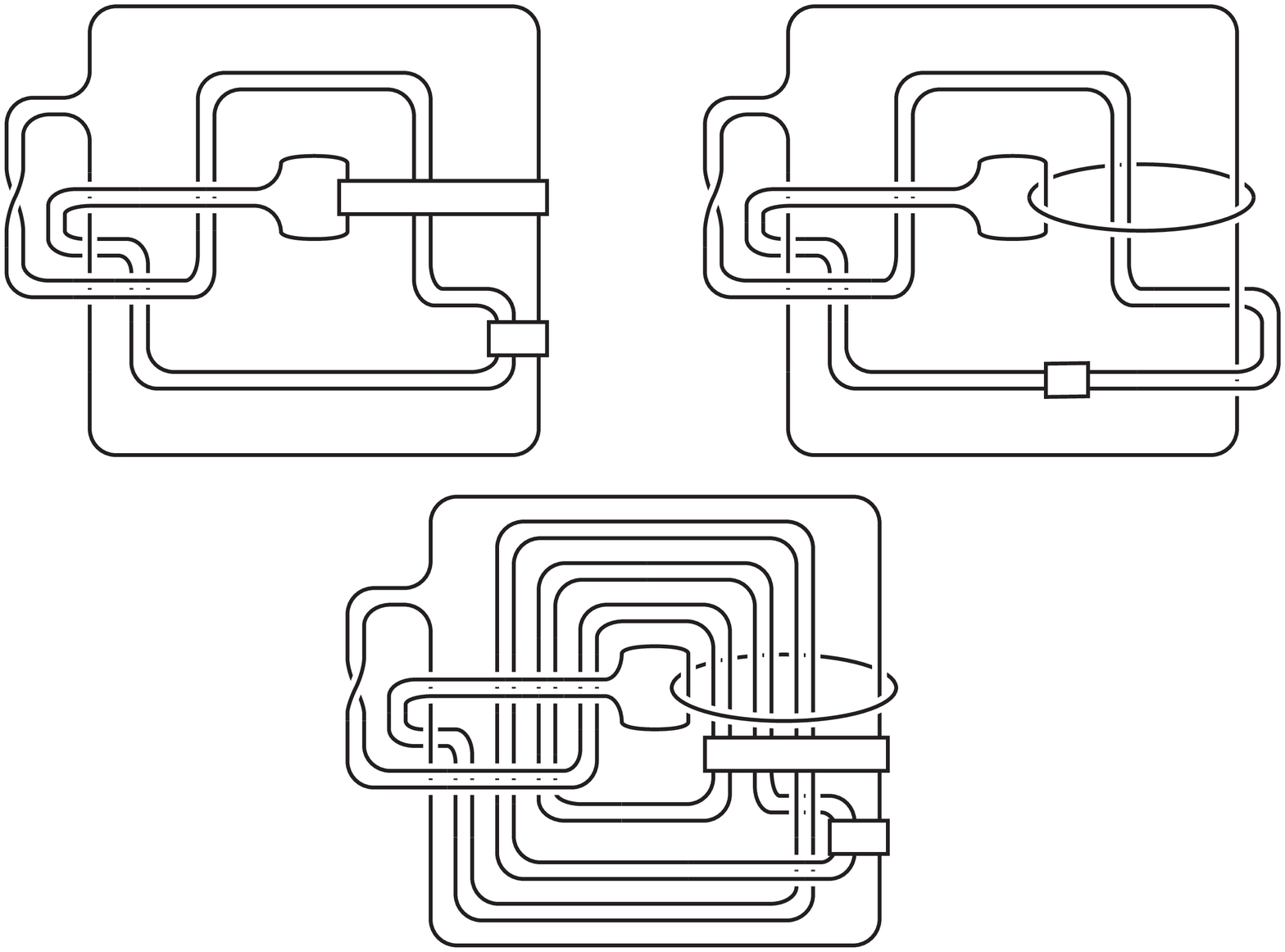}
\put(34,58){$1$}
\put(39.7,47){$1$}
\put(82.5,43.7){$1$}
\put(90,62){$-1$}
\put(70.5,20){$-1$}
\put(61.5,14.5){$1$}
\put(66.2,8){$1$}
\put(48,55){$=$}
\put(80,35.5){$J_1$}
\put(70,2){$J_2$}
\end{overpic}
\caption{An annulus presentation of the knot $J_2$ (lower half) obtained from $J_0$ 
by applying $(*1)$ two times. }
\label{fig:ExK2}
\end{figure}

The following lemma is obvious, however, important in our argument. 

\begin{lem}
Let $L$ be a $2$-component framed link 
which consists of $L_1$ with framing $(-1/n)$ and 
$L_2$ with framing $0$ as in the left side of Figure~\ref{fig:blow-down}. 
Suppose that the linking number of $L_1$ and $L_2$ is $\pm 1$ 
(with some orientation). 
Then the two Kirby diagrams in Figure~\ref{fig:blow-down} 
represent  the same $3$-manifold. 
\end{lem}

\begin{figure}[!htb]
\centering
\begin{overpic}[width=.4\textwidth]{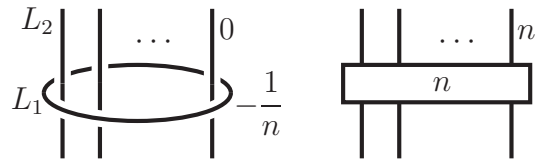}
\put(80,13.5){$n$}
\put(-5,10){$L_1$}
\put(-3,26){$L_2$}
\put(20,22){$\cdots$}
\put(81,22){$\cdots$}
\put(37,24){$0$}
\put(40,9){$-\dfrac1n$}
\put(97,24){$n$}
\end{overpic}
\caption{Two  Kirby diagrams represent  the same $3$-manifold. }
\label{fig:blow-down}
\end{figure}

\begin{thm} \label{thm:diffeo3}
Let $K$ be a knot with an annulus presentation 
and $K'$ be the knot obtained from $K$ by the operation $(*n)$.
Then 
\[ M_{K}(n) \approx M_{K'}(n) \, . \] 
\end{thm}
\begin{proof}
First, 
we consider the case where $K = J_0 = 8_{20}$ with the usual annulus presentation 
as in Figure~\ref{fig:*n}. 
Figure~\ref{fig:HM1} shows that 
$M_{K}(n) $ is represented by the last diagram in Figure~\ref{fig:HM1}, 
and this is diffeomorphic to $M_{K'}(n)$ by Figure~\ref{fig:HM2}. 
The moves in Figure~\ref{fig:HM2} correspond to the operation $(*n)$.

\begin{figure}[!htb]
\centering
\begin{overpic}[width=\textwidth]{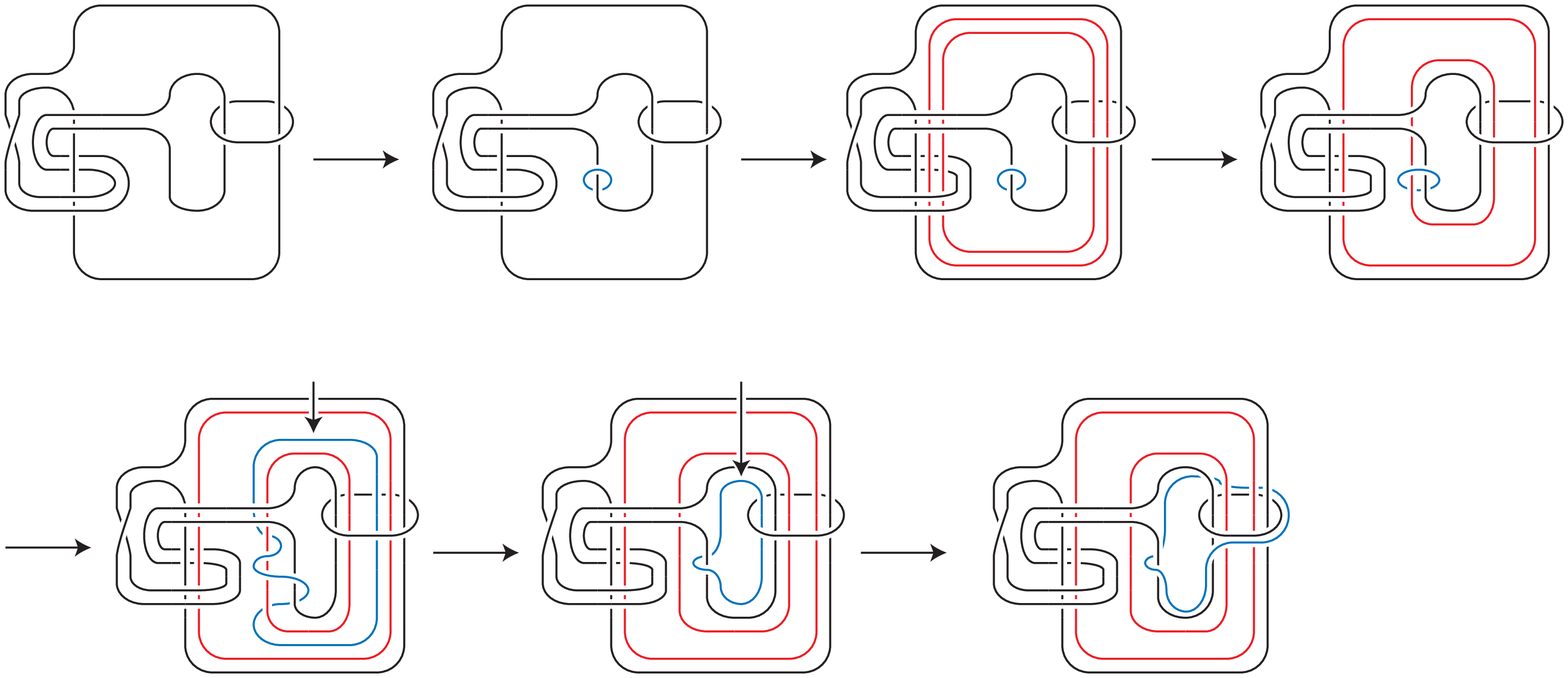}
\put(18.5,36){$-1$}
\put(46,36){$-1$}
\put(72.2,36){$-1$}
\put(99.5,36){$-1$}
\put(26.5,10.7){$-1$}
\put(53.7,10.7){$-1$}
\put(74.4,8){{\footnotesize $-1$}}
\put(18,42){$n$}
\put(45,42){$0$}
\put(71.5,42){$0$}
\put(99,42){$0$}
\put(26,16.5){$0$}
\put(53.3,16.5){$0$}
\put(81.2,16.5){$0$}
\put(34.5,28.5){\textcolor[cmyk]{1,0.5,0,0}{$-\frac{1}{n}$}}
\put(61,28.5){\textcolor[cmyk]{1,0.5,0,0}{$-\frac{1}{n}$}}
\put(86.5,28.5){\textcolor[cmyk]{1,0.5,0,0}{$-\frac{1}{n}$}}
\put(17,19.5){\textcolor[cmyk]{1,0.5,0,0}{$-1-\frac{1}{n}$}}
\put(45,19.5){\textcolor[cmyk]{1,0.5,0,0}{$-1-\frac{1}{n}$}}
\put(82.5,10){\textcolor[cmyk]{1,0.5,0,0}{$-\frac{1}{n}$}}
\put(58,37.2){\textcolor[cmyk]{0,1,1,0}{$1$}}
\put(60.5,37.2){\textcolor[cmyk]{0,1,1,0}{$-1$}}
\put(86,40){\textcolor[cmyk]{0,1,1,0}{$-1$}}
\put(95.4,30){\textcolor[cmyk]{0,1,1,0}{$1$}}
\put(13,15){\textcolor[cmyk]{0,1,1,0}{$-1$}}
\put(22.5,4){\textcolor[cmyk]{0,1,1,0}{$1$}}
\put(40.3,15){\textcolor[cmyk]{0,1,1,0}{$-1$}}
\put(50.5,4){\textcolor[cmyk]{0,1,1,0}{$1$}}
\put(69,15){\textcolor[cmyk]{0,1,1,0}{$-1$}}
\put(78.5,4){\textcolor[cmyk]{0,1,1,0}{$1$}}
\put(20,33.8){blow}
\put(21,31.5){up}
\put(73.8,33.8){slide}
\put(.5,9){slide}
\put(27,9){{\small isotopy}}
\put(55,8.8){slide}
\end{overpic}
\caption{A proof of  $M_{K}(n) \approx M_{K'}(n)$
when $K = 8_{20}$.}
\label{fig:HM1}
\end{figure}
\begin{figure}[!htb]
\centering
\begin{overpic}[width=.9\textwidth]{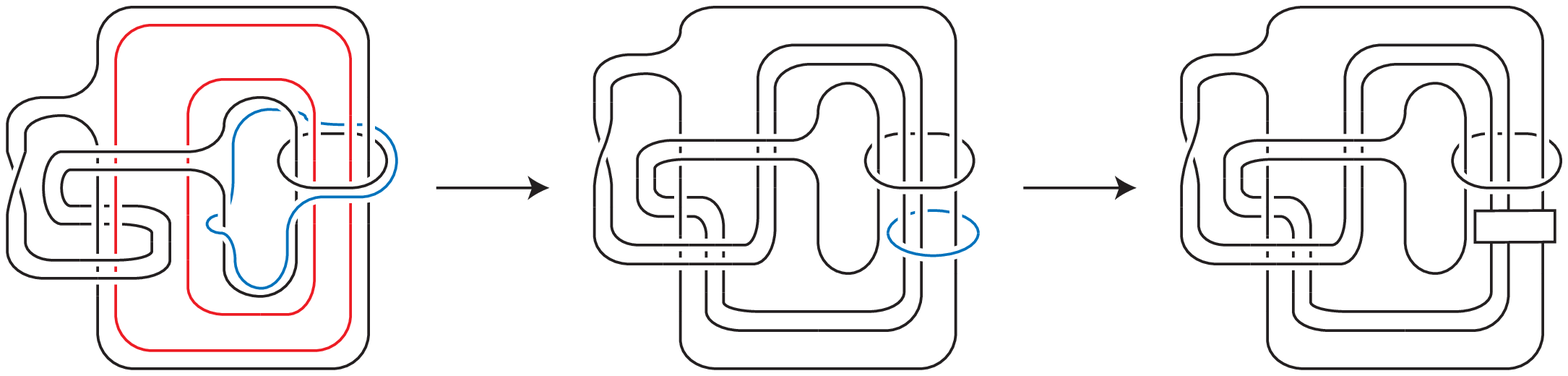}
\put(24,22){$0$}
\put(61,22){$0$}
\put(98.5,22){$n$}
\put(24,16){\textcolor[cmyk]{1,0.5,0,0}{$-\frac{1}{n}$}}
\put(62,7){\textcolor[cmyk]{1,0.5,0,0}{$-\frac{1}{n}$}}
\put(8,20){\textcolor[cmyk]{0,1,1,0}{$-1$}}
\put(20,7.2){\textcolor[cmyk]{0,1,1,0}{$1$}}
\put(15.2,11){{\footnotesize $-1$}}
\put(61.5,14.5){$-1$}
\put(99,14.5){$-1$}
\put(96,8.8){$n$}
\put(65.5,12.5){blow}
\put(65.5,9.7){down}
\end{overpic}
\caption{Moves which correspond to the operation $(*n)$.}
\label{fig:HM2}
\end{figure}

Next we consider the general  case.
Let $(A, b, c)$ be an annulus presentation of $K$. 
As seen in Figure~\ref{fig:HM3}, 
$M_{K}(n) $ is represented by the last diagram in Figure~\ref{fig:HM3}. 
Now it is not difficult to see that this is diffeomorphic to $M_{K'}(n)$. 
\begin{figure}[!htb]
\centering
\begin{overpic}[width=\textwidth]{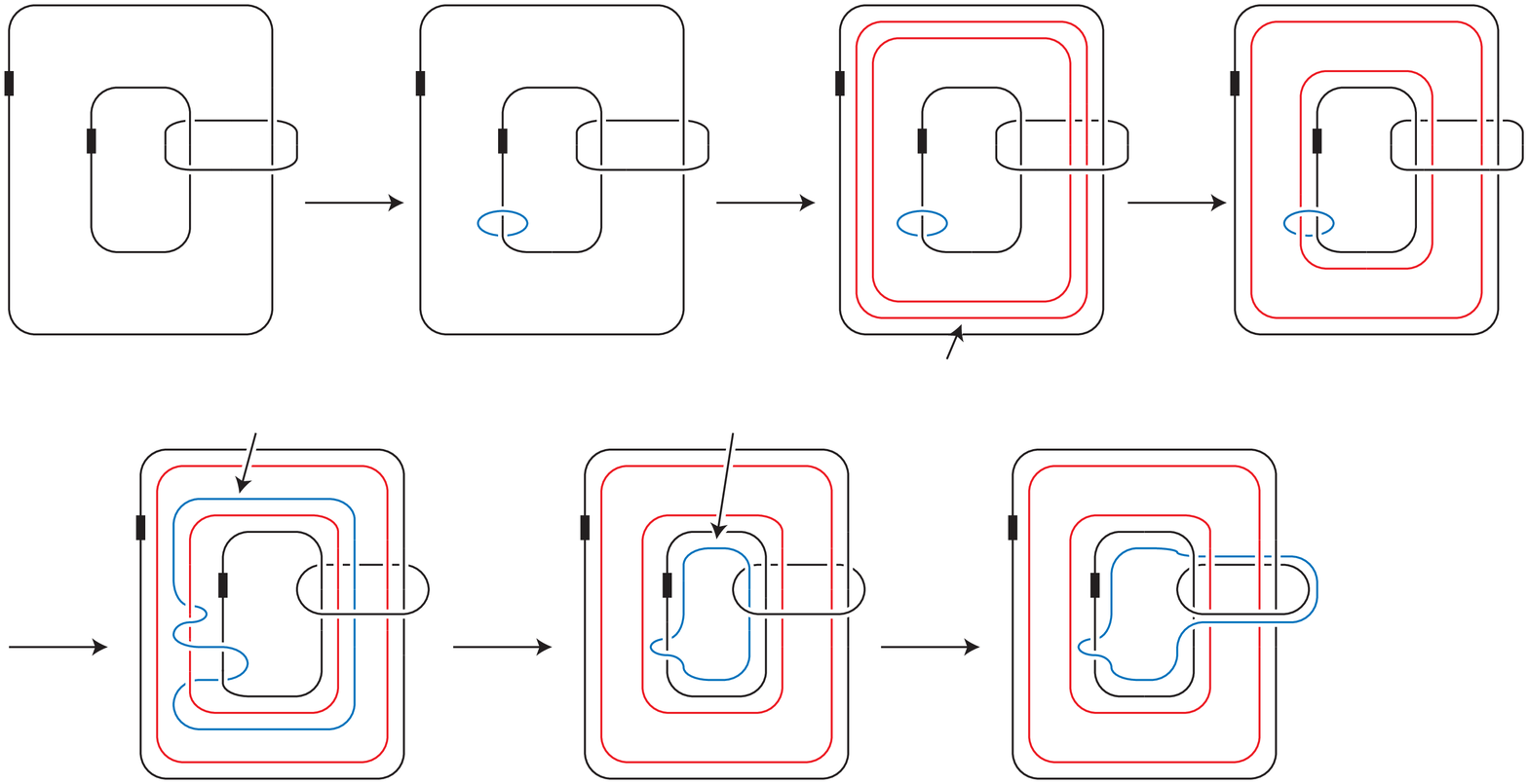}
\put(17.8,50){$n$}
\put(19.5,41.5){$-1$}
\put(44.8,50){$0$}
\put(46.8,41.5){$-1$}
\put(72.5,50){$0$}
\put(74.2,41.5){$-1$}
\put(98.5,50){$0$}
\put(100,41.5){$-1$}
\put(26.5,21){$0$}
\put(28,12){$-1$}
\put(55.5,21){$0$}
\put(57,12){$-1$}
\put(83.5,21){$0$}
\put(73.5,12){$-1$}
\put(34.5,36.5){\textcolor[cmyk]{1,.5,0,0}{$-\frac{1}{n}$}}
\put(62,36.5){\textcolor[cmyk]{1,.5,0,0}{$-\frac{1}{n}$}}
\put(87.3,36.5){\textcolor[cmyk]{1,.5,0,0}{$-\frac{1}{n}$}}
\put(12,24){\textcolor[cmyk]{1,.5,0,0}{$-1-\frac{1}{n}$}}
\put(43,24){\textcolor[cmyk]{1,.5,0,0}{$-1-\frac{1}{n}$}}
\put(86.5,12.5){\textcolor[cmyk]{1,.5,0,0}{$-\frac{1}{n}$}}
\put(61,32){\textcolor[cmyk]{0,1,1,0}{$-1$}}
\put(60.7,27){\textcolor[cmyk]{0,1,1,0}{$1$}}
\put(93,32){\textcolor[cmyk]{0,1,1,0}{$1$}}
\put(82.2,31){\textcolor[cmyk]{0,1,1,0}{$-1$}}
\put(12.6,15){\textcolor[cmyk]{0,1,1,0}{$1$}}
\put(10.5,1.7){\textcolor[cmyk]{0,1,1,0}{$-1$}}
\put(40.5,15){\textcolor[cmyk]{0,1,1,0}{$1$}}
\put(39.8,1.7){\textcolor[cmyk]{0,1,1,0}{$-1$}}
\put(68.6,15){\textcolor[cmyk]{0,1,1,0}{$1$}}
\put(67.8,1.7){\textcolor[cmyk]{0,1,1,0}{$-1$}}
\put(20,35.8){blow}
\put(20,33.5){up}
\put(74,35.8){slide}
\put(1,6.5){slide}
\put(29,6.5){isotopy}
\put(58,6.5){slide}
\end{overpic}
\caption{A proof of  $M_{K}(n) \approx M_{K'}(n)$ for the general case.}
\label{fig:HM3}
\end{figure}
\end{proof}

\begin{rem}\label{rem:simple_annulus_presentation}
Let $K$ be a knot with an annulus presentation $(A, b, c)$
and $K'$ be the knot obtained from $K$ by the operation ($*n$).
In general, $K'$ is much more complicated than $K$. 
If the annulus presentation $(A, b, c)$ is simple, then $K'$ is not too complicated. 
Indeed, let $(A, b_A, c)$ be the annulus presentation 
obtained from $(A, b, c)$ by applying the operation $(A)$ 
as in the left side of Figure~\ref{fig:RemOperation}. 
Then the knot $K'$ is indicated as in the right side of Figure~\ref{fig:RemOperation}.
\end{rem}

\begin{figure}[!htb]
\centering
\begin{overpic}[width=.7\textwidth]{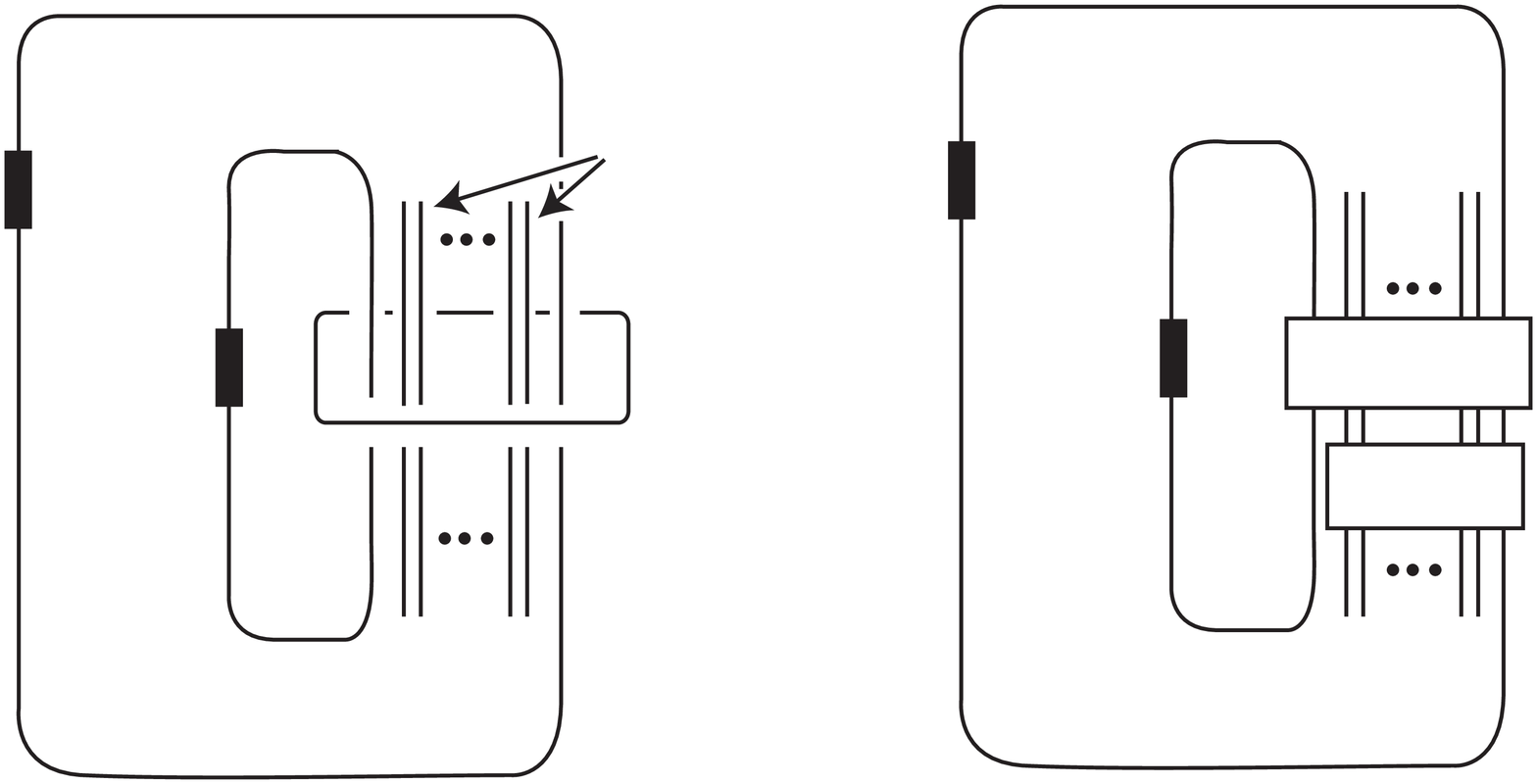}
\put(39.5,41){$b(I \times I)$}
\put(40.5,21){$c$}
\put(41.5,27){$-1$}
\put(89.5,26){$-1$}
\put(92,18.5){$n$}
\end{overpic}
\caption{The annulus presentation $(A, b_A, c)$ and the knot $K'$.}
\label{fig:RemOperation}
\end{figure}

\subsection{Extension of a diffeomorphism between 3-manifolds}\label{ssec:extension} 

In his seminal work, Cerf~\cite{Cerf} proved that 
any orientation preserving self diffeomorphism of $S^3$
extends to a self diffeomorphism of $B^4$.
As an application, Akbulut obtained the following lemma. 

\begin{lem}[\cite{Ak2}]\label{lem:extend}
Let $K$ and $K'$ be knots in $S^3 = \partial D^4$ 
with a diffeomorphism $g \colon \partial X_{K}(n) \to \partial X_{K'}(n)$, 
and let $\mu$ be a meridian of $K$. 
Suppose that
\begin{enumerate}
\item 
if $\mu$ is $0$-framed, then $g(\mu)$ is the $0$-framed unknot in the Kirby diagram representing $X_{K'}(n)$, and 
\item 
the Kirby diagram $X_{K'}(n) \cup h^1$  represents $D^4$, where $h^1$ is the $1$-handle represented by 
$g(\mu)$. 
\end{enumerate}
Then $g$ extends to a diffeomorphism $\widetilde{g} \colon X_{K}(n) \to X_{K'}(n)$ 
such that $\tilde{g}|_{\partial X_{K}(n)}=g$. 
\end{lem} 

This technique is called ``carving'' in~\cite{AkbulutBook}. 
For a proof, we refer the reader to \cite[Lemma 2.9]{AJOT}. 
Applying Lemma~\ref{lem:extend}, we show the following.

\begin{thm}\label{thm:diffeo4} 
Let $K$ be a knot with a simple annulus presentation 
and $K'$ be the knot obtained from $K$ by the operation $(*n)$. 
Then $X_{K}(n) \approx X_{K'}(n)$.  
\end{thm}
\begin{proof}
First, 
we consider the case where $K = 8_{20}$ with the usual simple annulus presentation.
Let $f \colon \partial X_{K}(n) \to  \partial X_{K'}(n)$ 
be the diffeomorphism given in Figures~\ref{fig:HM1} 
and \ref{fig:HM2}. 
Let $\mu$ be the meridian of $K$. 
If we suppose that $\mu$ is $0$-framed, 
then we can check that $f(\mu)$ is the $0$-framed unknot 
in the Kirby diagram of $X_{K'}(0)$ as in Figure~\ref{fig:Proof4-1}. 
Let $W$ be the $4$-manifold $D^4 \cup h^1 \cup h^2$, 
where $h^1$ is the dotted $1$-handle represented by $f(\mu)$ and $h^2$ 
is the $2$-handle represented by $K'$ with framing $n$.
Sliding $h^2$ over $h^1$, we obtain a canceling pair 
(see Figure~\ref{fig:Proof4-2}), implying that $W \approx B^4$.
By Lemma~\ref{lem:extend}, we have $\tilde{f} \colon X_{K}(0) \xrightarrow{\approx} X_{K'}(0)$. 

Next, we consider the general case.
Let $g \colon \partial X_{K}(n) \to  \partial X_{K'}(n)$ 
be the diffeomorphism given in the proof of Theorem~\ref{thm:diffeo3} in the general case  (see Figure~\ref{fig:Proof4-3}), 
and $\mu$ the meridian of $\partial X_{K}(n)$. 
In Figure~\ref{fig:Proof4-3}, 
the annulus presentation in the right hand side represents $K'$, 
see Remark~\ref{rem:simple_annulus_presentation}. 
If we suppose that $\mu$ is $0$-framed, 
then we can check that $g(\mu)$ is the $0$-framed unknot 
in the Kirby diagram of $X_{K'}(0)$ as in Figure~\ref{fig:Proof4-3}. 
Let $W$ be the $4$-manifold $D^4 \cup h^1 \cup h^2$, 
where $h^1$ is the dotted $1$-handle represented by $g(\mu)$ and $h^2$ 
is the $2$-handle represented by $K'$ with framing $n$. 
Sliding $h^2$ over $h^1$, we obtain a canceling pair 
(see Figure~\ref{fig:Proof4-4}), implying that $W \approx B^4$.
By Lemma~\ref{lem:extend} again, we have $\tilde{g} \colon X_{K}(0) \xrightarrow{\approx} X_{K'}(0)$.
\end{proof}

\begin{figure}[!htb]
\centering
\begin{overpic}[width=.75\textwidth]{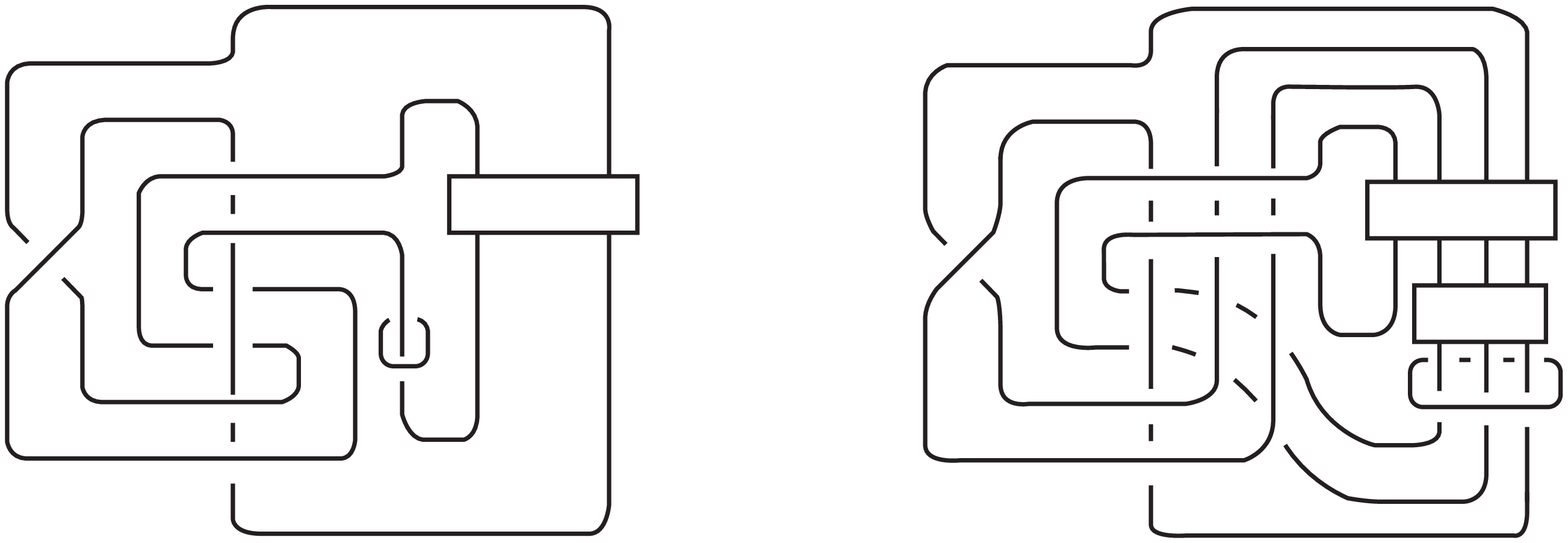}
\put(27.5,13){$0$}
\put(27,10){$\mu$}
\put(34,20.3){$1$}
\put(93,20){$1$}
\put(94,13.6){$n$}
\put(100,11){$0$}
\put(99,6){$f(\mu)$}
\put(48,15){$\approx$}
\put(48.5,17.7){$f$}
\end{overpic}
\caption{The image of $\mu$ under $f$.}
\label{fig:Proof4-1}
\end{figure}

\begin{figure}[!htb]
\centering
\begin{overpic}[width=.75\textwidth]{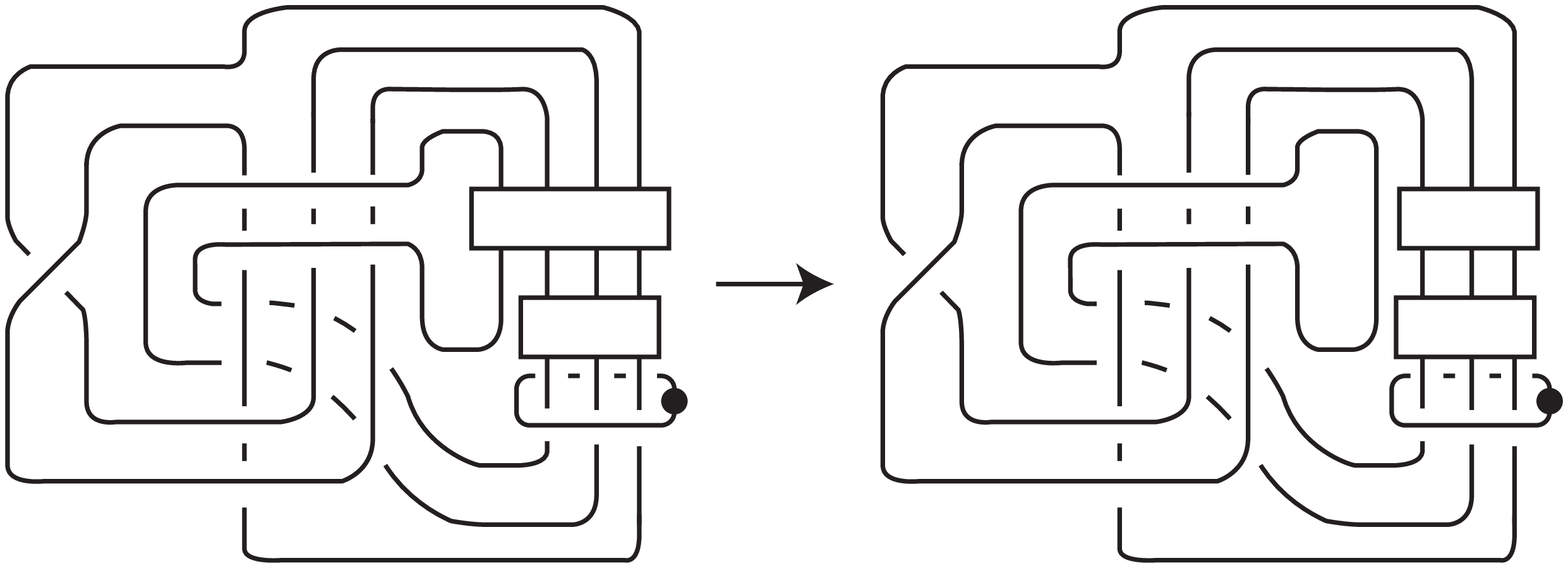}
\put(35.8,21){$1$}
\put(37,14.5){$n$}
\put(40.5,35){$n$}
\put(97,35){$n-2$}
\put(93.3,21){$1$}
\put(93,14.5){$n$}
\put(103,17){$\approx$}
\put(107,17){$B^4$}
\end{overpic}
\caption{The $4$-manifold $W$ is diffeomorphic to $B^4$.}
\label{fig:Proof4-2}
\end{figure}

\begin{figure}[!htb]
\centering
\begin{overpic}[width=.5\textwidth]{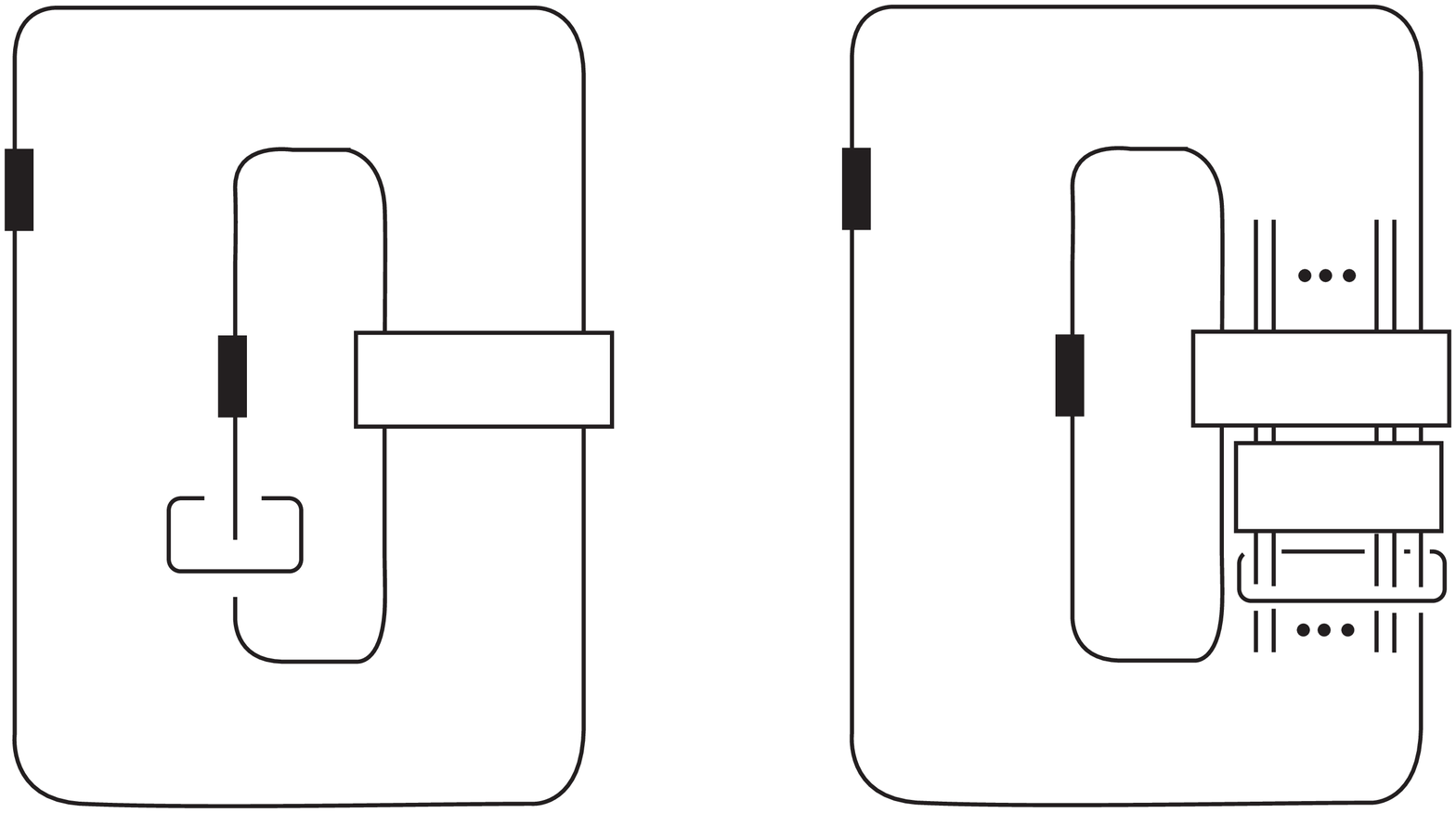}
\put(21,20){$0$}
\put(20.7,15){$\mu$}
\put(32.5,28){$1$}
\put(40,53){$n$}
\put(100.2,16.7){$0$}
\put(100,11.7){$g(\mu)$}
\put(90,28){$1$}
\put(91,21){$n$}
\put(98,53){$n$}
\put(48,27){$\approx$}
\end{overpic}
\caption{The image of $\mu$ under $g$.}
\label{fig:Proof4-3}
\end{figure}

\begin{figure}[!htb]
\centering
\begin{overpic}[width=.5\textwidth]{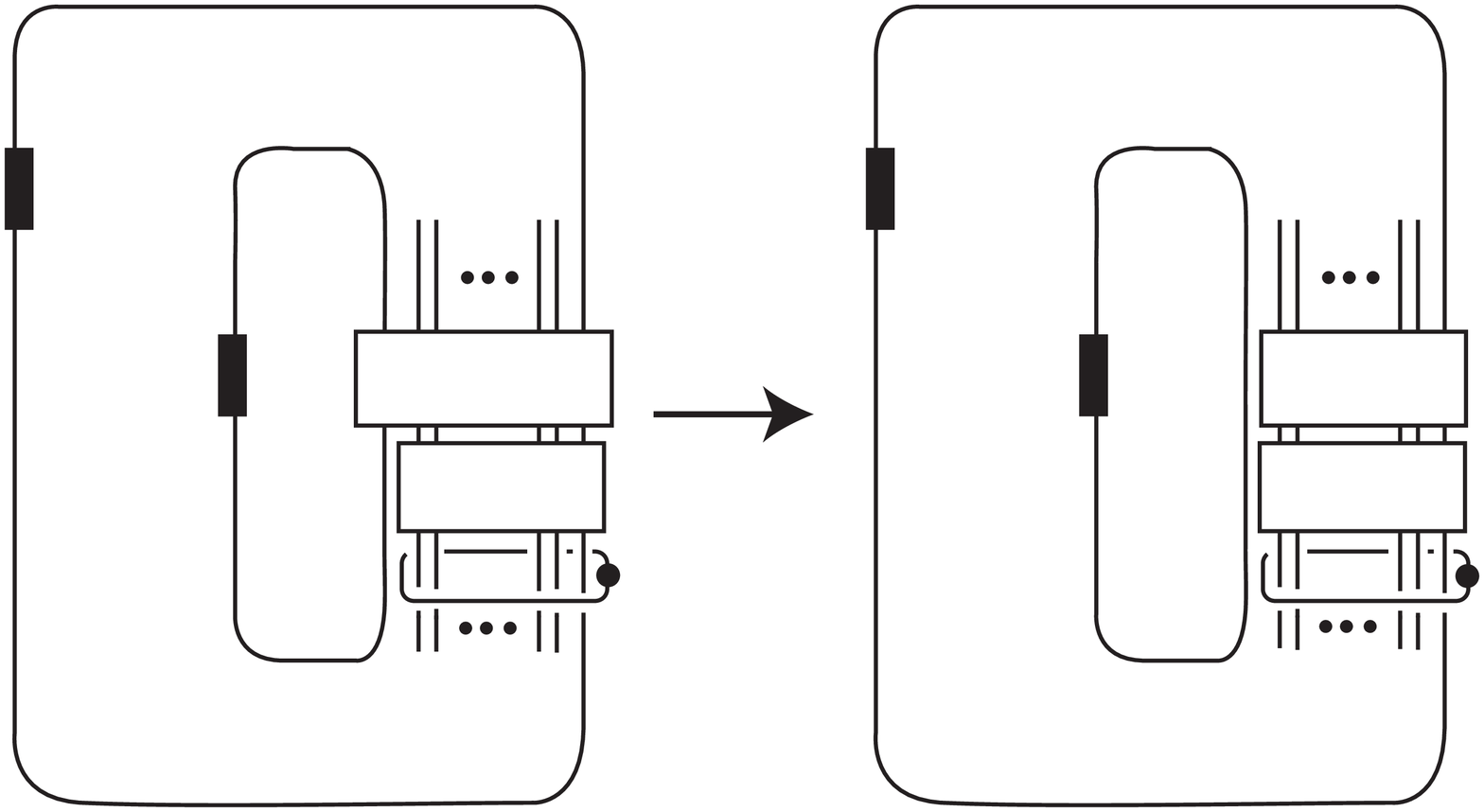}
\put(31.5,27.3){$1$}
\put(32.5,20.6){$n$}
\put(39.5,52){$n$}
\put(91,27.3){$1$}
\put(91,20.6){$n$}
\put(98,52){$n-2$}
\put(106,25){$\approx$}
\put(112,25){$B^4$}
\end{overpic}
\caption{The 4-manifold $W$ is diffeomorphic to $B^4$.}
\label{fig:Proof4-4}
\end{figure}

\begin{rem}
It would be interesting to characterize the knots 
which admit simple annulus presentations in terms of other topological properties.  It is known that a knot with unknotting number one 
admits a simple annulus presentation (see \cite[Lemma 2.2]{AJOT}). 
\end{rem}

\subsection{Proof of Theorem~\ref{thm:main}}\label{ssec:proof}

For a knot $K$, 
we denote by $\Delta_{K}(t)$ the Alexander polynomial of $K$.
We assume that $\Delta_K(1) = 1$ and $\Delta_{K}(t)$ is of the symmetric form 
\[ \Delta_{K}(t) = a_0 + \sum_{i=1}^d a_i(t^i + t^{-i}) \, . \]
We call the integer $d$ the {\em degree}\footnote{Usually, the degree is defined as $2d$. } of $\Delta_{K}(t)$, and denote it by $\deg \Delta_{K}(t)$. 
For example, $\deg (-1 + t + t^{-1}) = 1$. 

In this subsection, we define a ``good'' annulus presentation. 
Theorem~\ref{thm:main} will be shown as a typical case of the argument in this subsection. 
The following technical lemma plays an important role. 

\begin{lem}\label{lem:technical}
Let $n$ be a positive integer. 
Let $K$ be a knot with a good annulus presentation, 
and $K'$ be the knot obtained from $K$ by applying the operation $(*n)$. 
Then 
\begin{enumerate}[{\rm (i)}]
\item 
$K'$ also admits a good annulus presentation, and 
\item 
$\deg \Delta_K(t) < \deg \Delta_{K'}(t)$.
\end{enumerate}
\end{lem}

We will prove Lemma~\ref{lem:technical} later. 
Using Lemma~\ref{lem:technical}, we show the following 
which yields Theorem~\ref{thm:main} as an immediate corollary. 

\begin{thm}\label{thm:main2}
Let $n$ be a positive integer. 
Let $K_{0}$ be a knot with a good annulus presentation 
and $K_{i}$ $( i \ge 1)$ the knot obtained from $K_{i-1}$ by applying the operation $(*n)$. 
Then 
\begin{enumerate}
\item 
$X_{K_0}(n) \approx X_{K_1}(n) \approx X_{K_2}(n)   \approx \cdots$, and 
\item 
the knots $K_{0}, K_{1}, K_{2}, \dots$ are mutually distinct.
\end{enumerate}
Let $\overline{K}_{i}$ be the mirror image of $K_{i}$. 
Then 
\begin{enumerate}
\setcounter{enumi}{2}
\item 
$X_{\overline{K}_{0}}(-n) \approx X_{\overline{K}_{1}}(-n) \approx X_{\overline{K}_{2}}(-n) \approx \cdots$, and 
\item 
the knots $\overline{K}_{0}, \overline{K}_{1}, \overline{K}_{2}, \dots$ are mutually distinct.
\end{enumerate}
\end{thm}
\begin{proof}
By the definition (Definition~\ref{def:good}), 
any good annulus presentation is simple. 
Thus, by Theorem~\ref{thm:diffeo4}, we have 
\[ X_{K_0}(n) \approx X_{K_1}(n) \approx X_{K_2}(n)   \approx \cdots \, .\]
By Lemma ~\ref{lem:technical} (i),
each $K_{i}$ $( i \ge 1)$ also admits a good annulus presentation. 
Thus, by Lemma~\ref{lem:technical} (ii), we have 
\[ \deg \Delta_{K_0}(t)  < \deg \Delta_{K_1}(t)  < \deg \Delta_{K_2}(t) < \cdots \, . \]
This implies that the knots $K_{0}, K_{1}, K_{2}, \dots$ are mutually distinct. 

Since $X_{K_i}(n) \approx X_{\overline{K}_i}(-n)$ 
and $\deg \Delta_{K_i}(t)=\deg \Delta_{\overline{K}_i}(t)$, 
we have
\[ X_{\overline{K}_0}(-n) \approx X_{\overline{K}_1}(-n) \approx X_{\overline{K}_2}(-n)   \approx \cdots \, , \  \text{and}\]
\[ \deg \Delta_{\overline{K}_0}(t)  < \deg \Delta_{\overline{K}_1}(t)  < \deg \Delta_{\overline{K}_2}(t) < \cdots \, .\]
This completes the proof of Theorem~\ref{thm:main2}. 
\end{proof}

\subsubsection{Good annulus presentation and the Alexander polynomial}

Let $K$ be a knot with a simple annulus presentation $(A, b, c)$. 
Note that the knot $(\partial A \setminus b(\partial I \times I) )\cup b(I \times \partial I)$ is 
trivial\footnote{$K$ is the knot $(\partial A \setminus b(\partial I \times I) )\cup b(I \times \partial I)$ in $M_c(-1)$.} in $S^3$ 
if we ignore the $(-1)$-framed loop $c$. 
We denote by $U$ this trivial knot. 
Since $(A,b, c)$ is simple, $U \cup c$ can be isotoped so that 
$U$ bounds a ``flat'' disk $D$ (contained in $\R^2 \cup \{\infty\}$). 
This isotopy, denoted by $\varphi_b$, 
is realized by shrinking the band $b(I \times I)$. 
For simplicity, the isotopy $\varphi_b$ is also denoted by $\varphi$. 
For example, see Figure~\ref{fig:ExStandard}. 
In the abbreviated form, 
$\varphi$ is represented as in Figure~\ref{fig:StandardAbb}. 
Here we note that the linking number of $U$ and $c$ is zero 
since we assumed that $A \cup b(I \times I)$ is orientable. 
Let $\Sigma$ be the disk bounded by $c$ as in Figure~\ref{fig:StandardAbb}. 
We assume that $\Sigma$ remains fixed through the isotopy $\varphi$. 

\begin{figure}[!htb]
\centering
\begin{overpic}[width=.75\textwidth]{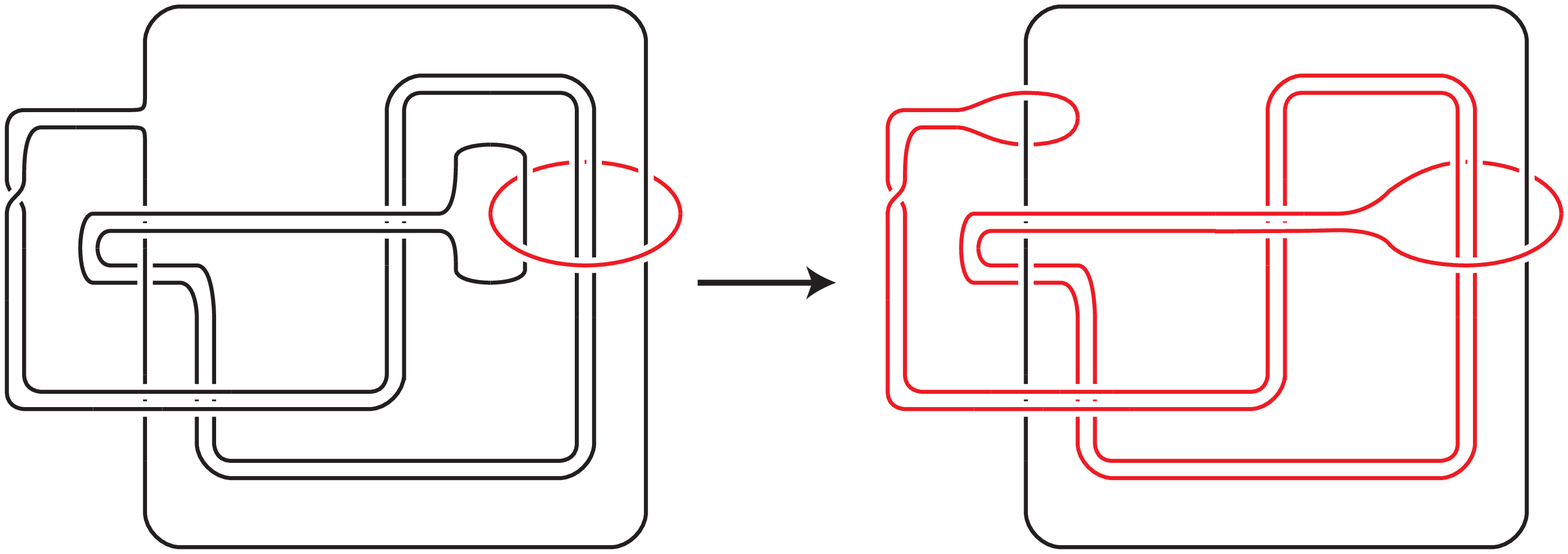}
\put(43,23.5){$\textcolor[cmyk]{0,1,1,0}{c}$}
\put(41,34){$U$}
\put(82.5,13){$\textcolor[cmyk]{0,1,1,0}{c}$}
\put(97.5,34){$U$}
\put(43.8,19){isotopy}
\put(47.5,14.5){$\varphi$}
\end{overpic}
\caption{By the isotopy $\varphi$ (shrinking the band $b(I \times I)$), 
$U \cup c$ (the left side) is changed to the right side.}
\label{fig:ExStandard}
\end{figure}
\begin{figure}[!htb]
\centering
\begin{overpic}[width=.75\textwidth]{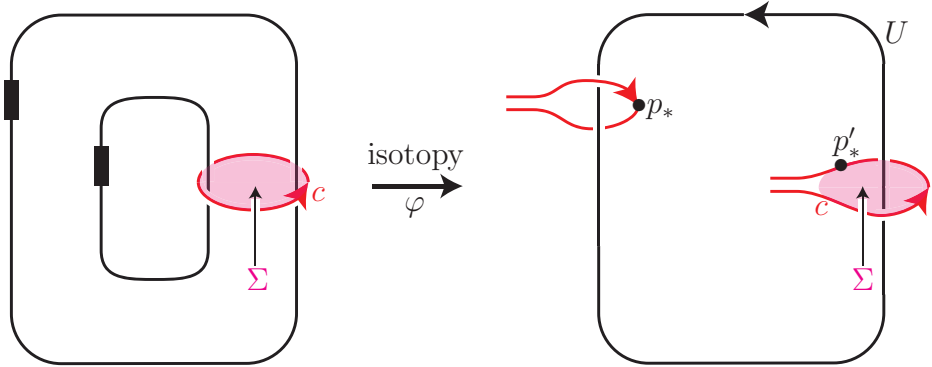}
\put(33,18){$\textcolor[cmyk]{0,1,1,0}{c}$}
\put(87,16.5){$\textcolor[cmyk]{0,1,1,0}{c}$}
\put(94.5,35){$U$}
\put(39,22){isotopy}
\put(43,17){$\varphi$}
\put(67.3,27.5){$\bullet$}
\put(89,21){$\bullet$}
\put(69,27.5){$p_*$}
\put(89,23.5){$p_*'$}
\put(26,8.5){\textcolor[cmyk]{0,1,0,0}{$\Sigma$}}
\put(91,8.5){\textcolor[cmyk]{0,1,0,0}{$\Sigma$}}
\end{overpic}
\caption{The isotopy $\varphi$ in the abbreviated form of $(A,b,c)$.
}
\label{fig:StandardAbb}
\end{figure}

After the isotopy $\varphi$, cutting along the disk $D$, 
the loop $c$ is separated into arcs whose endpoints are in $D$. 
Furthermore, choosing orientations on $c$ and $U$, these arcs are oriented. 
We choose the orientations on $c$ and $U$ as in Figure~\ref{fig:StandardAbb}, 
and orientations on $D$ and $\Sigma$ consistent with those on $c$ and $U$. 
These oriented arcs are classified into four types as follows: 
For $p \in c \cap D$, let $\sign(p) = \pm$ 
according to the sign of the intersection between $D$ and $c$ at $p$. 
For an oriented arc $\alpha$, let $p_s$ (resp.~$p_t$) 
be the starting point (resp.~terminal point) of $\alpha$. 
Then we say that $\alpha$ is of {\em type} $(\sign(p_s) \sign(p_t))$. 
That is, the oriented arc $\alpha$ is of type  $(++)$, $(--)$, $(+-)$, or $(-+)$. 
For example, see Figure~\ref{fig:ExStandard2}.

\begin{figure}[!htb]
\centering
\begin{overpic}[width=.4\textwidth]{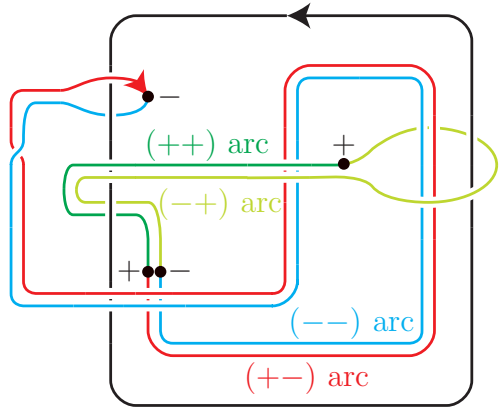}
\put(30.6,61.7){$-$}
\put(65.7,52){$+$}
\put(32.7,26.5){$-$}
\put(22.5,26.5){$+$}
\put(48,5){\textcolor[cmyk]{0,1,1,0}{$(+-)$ arc}}
\put(57,16){\textcolor[cmyk]{1,0,0,0}{$(--)$ arc}}
\put(30.5,40){\textcolor[cmyk]{.3,0,1,0}{$(-+)$ arc}}
\put(28,52){\textcolor[cmyk]{1,0,1,0}{$(++)$ arc}}
\end{overpic}
\caption{The four types of arcs. }
\label{fig:ExStandard2}
\end{figure}

Let $E(U)$ be the exterior of $U$ and $\tilde E(U)$ its infinite cyclic cover. 
Notice that $\tilde E(U)$ consists of infinitely many copies of a cylinder 
obtained from $E(U)$ by cutting along $D$. 
Thus $\tilde E(U)$ is diffeomorphic to $D \times \R \approx \cup_{i \in \Z} \left( D \times [i,i+1] \right)$. 
Each oriented arc is lifted in $\tilde E(U)$ as shown in Figure~\ref{fig:LiftOfArc}. 
Note that each arc can be knotted in Figure~\ref{fig:LiftOfArc}. 
Hereafter, for simplicity, we say an arc instead of an oriented arc.

\begin{figure}[!htb]
\bigskip
\begin{overpic}[width=.7\textwidth]{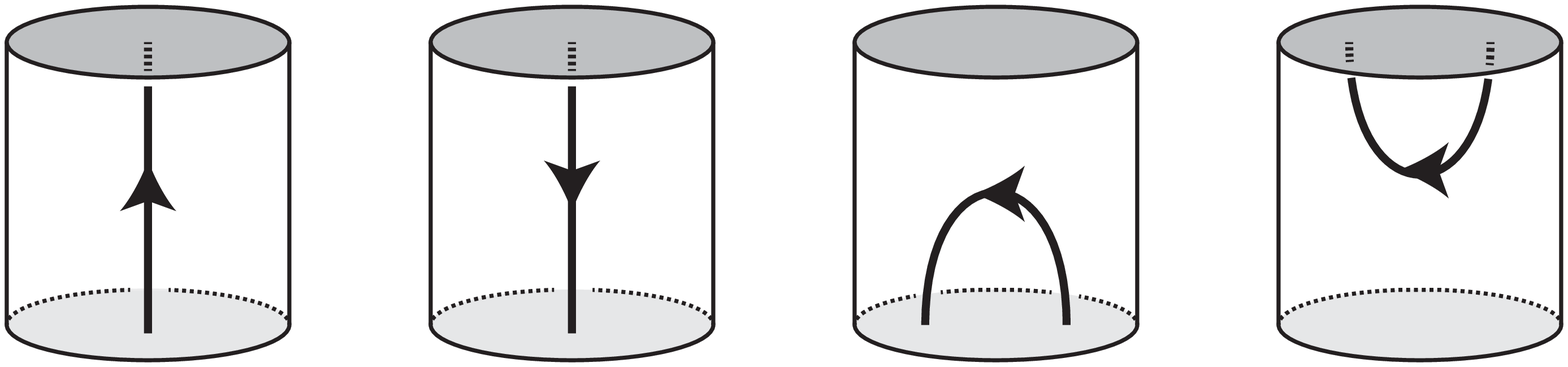}
\put(3,-3.5){$D^2 \times \{ i\}$}
\put(0,24){$D^2 \times \{ i + 1\}$}
\end{overpic}
\medskip
\caption{Lifts of oriented arcs of type $(++), (--), (+-)$, and $(-+)$ respectively. }
\label{fig:LiftOfArc}
\end{figure}
\begin{figure}[!htb]
\begin{overpic}[width=.3\textwidth]{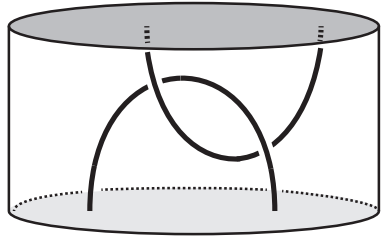}
\end{overpic}
\caption{Lifts of the arcs of type $(+-)$ and $(-+)$ from a good annulus presentation. }
\label{fig:LiftOfArc2}
\end{figure}

\begin{defn}\label{def:good}
We say that a simple annulus presentation $(A, b, c)$ is {\em good} 
if $b(I \times \partial I) \cap \mathrm{int}A \ne \emptyset$
and 
 the set of  arcs $\mathcal A$ obtained as above satisfies the following up to isotopy. 
\begin{enumerate}
\item 
$\mathcal A$ contains just one $(+-)$ arc and one $(-+)$ arc, 
and they are lifted as in Figure~\ref{fig:LiftOfArc2}, 
that is, the linking number of the arcs rel $D \times \{i, i+1\}$ is $\pm 1$. 
Here each of the two arcs is possibly itself knotted. 
\item 
For $\alpha \in \mathcal A$, 
if $\alpha \cap \mathrm{int}\Sigma \ne \emptyset$, 
then $\alpha$ is of type $(++)$ (resp.~$(--)$) 
and the sign of each intersection point in 
$\mathrm{int}\Sigma \cap \alpha$ is $+$ (resp.~$-$). 
 
\end{enumerate}
\end{defn}

\begin{rem}\label{rem:good}
For a simple annulus presentation $(A,b,c)$, after the isotopy $\varphi$, 
the intersection $c \cap D$ corresponds to 
the intersection $b(I \times \partial I) \cap \mathrm{int} A$ 
and further two points $p_*$ and $p_*'$ depicted in Figure~\ref{fig:StandardAbb}. 
Notice that 
\[ b(I \times \partial I) \cap \mathrm{int} A = \sqcup_i \, b(\{t_i\} \times \partial I) \]  
for some $0 < t_1 < \dots < t_r < 1$. 
For each $i$, $b(\{t_i\} \times \partial I)$ consists of two points whose signs differ. 
Furthermore, with the orientation as in Figure~\ref{fig:StandardAbb}, we have 
\[ \sign(p_*) = - \, \text{ \, and \, } \sign(p_*') = + \, . \] 
\end{rem}

\begin{ex}\label{ex:GoodIsotopy}
The annulus presentation $(A, b, c)$ of the knot $J_0 = 8_{20}$ is good since 
it is changed by the isotopy $\varphi = \varphi_b$ as in Figure~\ref{fig:ExStandard4}. 
Applying the operation $(A)$, 
we obtain the annulus presentation as in Figure~\ref{fig:ExStandard}. 
We denote this annulus presentation by $(A, b_A, c)$. 
Now we check that $(A, b_A, c)$ is good. 
It is obvious that $b_A(I \times \partial I) \cap \mathrm{int}A \ne \emptyset$. 
As in the left side of Figure~\ref{fig:ExStandard3}, 
after the isotopy $\varphi = \varphi_{b_A}$, 
the set of arcs satisfies condition (1) of Definition~\ref{def:good}. 
However the $(+-)$ arc intersects $\mathrm{int} \Sigma$, 
that is, condition (2) of Definition~\ref{def:good} does not hold. 
In such a case, changing the position of an intersection 
as in Figure~\ref{fig:ExStandard3} by applying an isotopy to the $(+-)$ arc, 
we obtain the set of arcs satisfying condition (2). 
Note that after this isotopy we can assume that 
the $(+-)$ and $(-+)$ arcs are fixed by a subsequent application of the operation $(A)$. 
\begin{figure}[!htb]
\centering
\begin{overpic}[width=.8\textwidth]{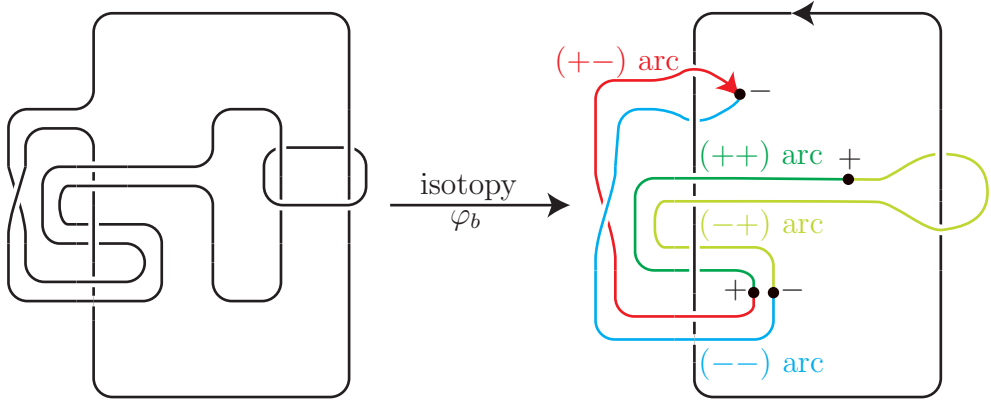}
\put(75,30.5){$-$}
\put(78.2,10.5){$-$}
\put(72.5,10.5){$+$}
\put(84,23.7){$+$}
\put(55.5,33.5){\textcolor[cmyk]{0,1,1,0}{$(+-)$ arc}}
\put(70,3){\textcolor[cmyk]{1,0,0,0}{$(--)$ arc}}
\put(70,17){\textcolor[cmyk]{.3,0,1,0}{$(-+)$ arc}}
\put(70,24){\textcolor[cmyk]{1,0,1,0}{$(++)$ arc}}
\put(42,21){isotopy}
\put(45,18){$\varphi_b$}
\end{overpic}
\caption{The annulus presentation $(A, b, c)$ of $J_0 = 8_{20}$ is good.}
\label{fig:ExStandard4}
\end{figure}
\begin{figure}[!htb]
\centering
\begin{overpic}[width=.9\textwidth]{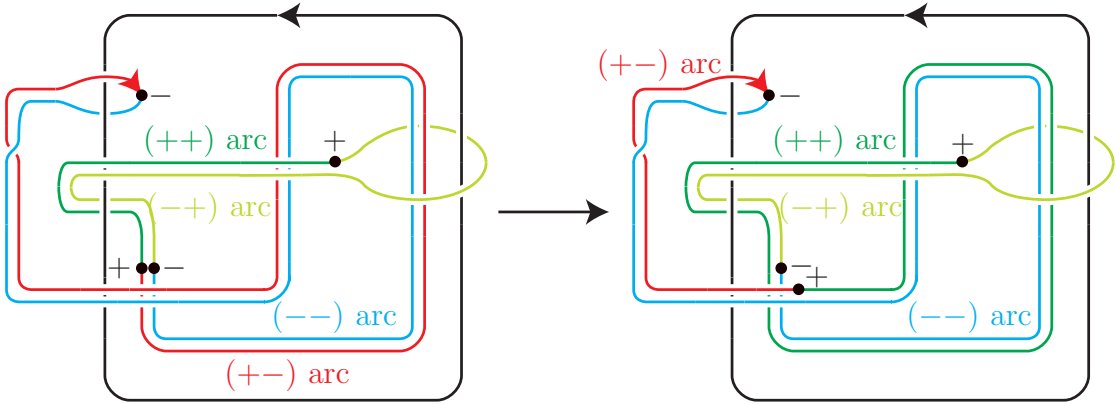}
\put(13.2,27){$-$}
\put(28.8,23){$+$}
\put(14.4,11.6){$-$}
\put(9.5,11.6){$+$}
\put(20,2){\textcolor[cmyk]{0,1,1,0}{$(+-)$ arc}}
\put(24.2,7.2){\textcolor[cmyk]{1,0,0,0}{$(--)$ arc}}
\put(13,17){\textcolor[cmyk]{.3,0,1,0}{$(-+)$ arc}}
\put(12.7,23){\textcolor[cmyk]{1,0,1,0}{$(++)$ arc}}
\put(69.5,27){$-$}
\put(85,22.7){$+$}
\put(70.5,11.8){$-$}
\put(71.9,10.8){$+$}
\put(53.3,29.8){\textcolor[cmyk]{0,1,1,0}{$(+-)$ arc}}
\put(81,7.2){\textcolor[cmyk]{1,0,0,0}{$(--)$ arc}}
\put(69.5,17){\textcolor[cmyk]{.3,0,1,0}{$(-+)$ arc}}
\put(69,23){\textcolor[cmyk]{1,0,1,0}{$(++)$ arc}}
\end{overpic}
\caption{By an isotopy, we move the intersection point of $c \cap D$. }
\label{fig:ExStandard3}
\end{figure}
\end{ex}

Let $E(K)$ be the exterior of a knot $K$. 
Considering a surgery description of the infinite cyclic covering, $\tilde E(K)$, 
of $E(K)$, we have the following. 

\begin{lem}\label{lem:LiftOfArc} 
If a knot $K$ admits a good annulus presentation, then 
\begin{align}\label{eq:ArcDeg}
\deg \Delta_K(t) = \# \{ \text{arcs of type } (++) \} + 1 \, .
\end{align}
\end{lem}

\begin{proof}[Proof] 
Let $(A, b, c)$ be a good annulus presentation of $K$.
Then $K$ is represented by the unknot $U$ (in $M_{c}(-1)$).
After the isotopy $\varphi_b$, 
$U$ bounds a ``flat'' disk $D$ (contained in $\R^2 \cup \{\infty\}$)
as in Figure~\ref{fig:ExStandard}. 

Let $V$ be a closed tubular neighborhood of $c$ 
(which is diffeomorphic to a solid torus) 
and $J$ a simple closed curve of slope $-1$ on $\partial V$. 
Since $c$ and $U$ have linking number zero, 
we can construct the infinite cyclic cover $\tilde{E}(K)$ 
from $\tilde{E}(U) \approx \cup_{i \in \Z} \left( D \times [i,i+1] \right)$ as follows: 
Remove the interior of solid tori $\tilde{V}_{i}$ ($i \in \Z$) 
which lie above $V$ in $\tilde{E}(U)$, 
and sew solid tori $W_i$ ($i \in \Z$) back so that 
each meridian, $\mu_i$, of $W_i$ is attached to the lift, $\tilde{J}_i$, of $J$. 

Let $\Delta_{K}(t) = a_0 + \sum_{i=1}^{\infty} a_i(t^i + t^{-i})$. 
Then it is not difficult to see that  $a_i$ $(i=1,2, \dots)$ is 
the linking number between $\tilde{J}_{0}$ and the core of $\tilde{V}_{i}$
with suitable linking convention in $\tilde{E}(U)$. 
Note that $a_0 = 1 - 2 \sum_{i=1}^{\infty} a_i$ since $\Delta_K(1) = 1$. 
By condition (1) in Definition~\ref{def:good}, 
\[ 
|a_i| = \begin{cases}
1 & \text{if } i=\# \{ \text{arcs of type } (++) \} + 1 \, ,\\
0 & \text{if } i > \# \{ \text{arcs of type } (++) \} + 1 \, .
\end{cases}
\]
This implies that 
$\deg \Delta_{K}(t)=\# \{ \text{arcs of type } (++) \} + 1$.

For the details of a surgery description of $\tilde{E}(K)$ and the Alexander polynomial, 
we refer the reader to Rolfsen's book~\cite[Chapter 7]{Rolfsen}. 
\end{proof}

\begin{rem}
To show Lemma~\ref{lem:LiftOfArc}, 
we do not need conditions (2) and (3) in Definition~\ref{def:good}. 
These conditions are used to prove Lemma~\ref{lem:technical}. 
\end{rem}

\begin{rem}\label{rem:monic}
If a knot $K$ admits a good annulus presentation, 
then we can see that $\Delta_K(t)$ is monic. 
\end{rem}

Now we are ready to prove the main result in this section. 

\begin{proof}[Proof of Theorem~\ref{thm:main}]
The case where $n=0$ was proved in \cite{AJOT}. 
We can check that the simple annulus presentation of the knot $8_{20}$ 
as given in Figure~\ref{fig:ExStandard4} is good, see Example~\ref{ex:GoodIsotopy}. 
Thus the proof for the case where $n \neq 0$ is obtained by Theorem~\ref{thm:main2} immediately. 
\end{proof}

\subsubsection{Proof of Lemma~\ref{lem:technical}}

We start the proof of Lemma~\ref{lem:technical}. 
Let $(A, b, c)$ be a good annulus presentation of a knot $K$. 
Recall that the operation $(*n)$ is a composition of the two operations $(A)$ and $(T_n)$ 
for an annulus presentation. 
Let $(A, b_A, c)$ be the annulus presentation obtained from $(A, b, c)$ 
by applying the operation $(A)$, 
and $(A, b', c)$ the annulus presentation obtained from $(A, b_A, c)$ 
by applying the operation $(T_n)$. 
That is,
\[ (A, b, c) \overset{(A)}{\longrightarrow }  (A, b_A, c)  \overset{(T_n)}{\longrightarrow } (A, b', c). \]
Note that $K'$ admits the annulus presentation $(A, b', c)$.

First we show that $(A, b_A, c)$ is good. 
It is obvious that $b_A(I \times \partial I) \cap \mathrm{int}A \ne \emptyset$.
The operation $(A)$ preserves the number of arcs and type of each arc. 
We can suppose that the $(+-)$ arc and the $(-+)$ arc are fixed by the operation $(A)$ up to isotopy as discussed in Example~\ref{ex:GoodIsotopy}. 
Therefore the set of the arcs $\mathcal A$ (obtained from $(A, b_A, c)$) 
satisfies condition (1). 
Furthermore we can show that $\mathcal A$ satisfies condition (2) 
since the orientations of $c$ and $U$ are consistent with the operation $(A)$. 
Therefore $(A, b_A, c)$ is good. 

Next we show that $(A, b', c)$ is good. 
It is obvious that $b'(I \times \partial I) \cap \mathrm{int}A \ne \emptyset$.
The operation $(T_n)$ may increase the number of arcs. 
Indeed a $(++)$ (resp.~$(--)$) arc through $\Sigma$ is changed to $n+1$ $(++)$ (resp.~$(--)$) arcs since $(A, b_A, c)$ is good and $n > 0$, 
in particular, a $(++)$ arc (resp.~$(--)$ arc) intersects $\Sigma$ 
positively (resp.~negatively). 
Note that the $(+-)$ arc and the $(-+)$ arc are fixed by the operation $(T_n)$ 
since they are disjoint from $\Sigma$. 
Hence $(+-)$ arcs and $(-+)$ arcs are not produced by the operation $(T_n)$. 
Therefore the set of the arcs $\mathcal A'$ (obtained from $(A, b', c)$) 
 satisfies conditions (1) and (2). 
Hence $(A, b',c)$ (of $K'$) is good. 
This completes the proof of the claim (i) of Lemma~\ref{lem:technical}. 

Let $ \delta = \# \left(A \cap b(I \times \partial I) \right) / 2$ and 
$\sigma = \# \left( \Sigma \cap b(I \times \partial I) \right) / 2$. 
Then we see that 
\[ \# (A \cap b_A (I \times \partial I)) /2 = \delta \, , \qquad 
\# (\Sigma \cap b_A(I \times \partial I)) /2 = \sigma + \delta \, . \]
Then we have 
\begin{align*}
\# (A \cap b'(I \times \partial I)) /2 
& = \# (A \cap b_A(I \times \partial I))/2 + n \cdot \# (\Sigma \cap b_A(I \times \partial I)) /2 \\ 
&= (n+1) \delta + n \sigma \, , 
\end{align*}
and 
\begin{align*}
\# (\Sigma \cap b'(I \times \partial I)) 
&= \# (\Sigma \cap b_A(I \times \partial I)) \, .
\end{align*}
These are equivalent to 
\begin{align*}
\begin{pmatrix}
\delta' \\ 
\sigma'
\end{pmatrix}
= 
\begin{pmatrix}
n+1 & n \\ 
1 & 1 
\end{pmatrix}
\begin{pmatrix}
\delta \\ 
\sigma
\end{pmatrix},  
\end{align*}
where $ \delta' = \# \left(A \cap b'(I \times \partial I) \right) / 2$ and 
$\sigma' = \# \left( \Sigma \cap b'(I \times \partial I) \right) / 2$. 
Since $n \ge 1$ and $\delta \ge 1$, we have 
\begin{align}\label{eq:AlphaIneq}
\delta < \delta' \, . 
\end{align}
By the condition that $(A, b, c)$ and $(A, b',c)$ is good, 
and by Remark~\ref{rem:good}, we see that 
\[ \delta = \# \set{(++)\text{ arcs of }(A, b, c)} \, , \qquad 
\delta' = \# \set{(++)\text{ arcs of }(A, b', c)} \, . \] 
Therefore, by Lemma~\ref{lem:LiftOfArc}, we have 
\begin{align}\label{eq:DegSing}
\deg \Delta_K = \delta + 1 \, ,  \quad \deg \Delta_{K'} = \delta' + 1 \, . 
\end{align}
By \eqref{eq:AlphaIneq} and \eqref{eq:DegSing}, we have 
$\deg \Delta_{K}(t) < \deg \Delta_{K'}(t)$. 
This completes the proof of the claim (ii) of Lemma~\ref{lem:technical}, 
and thus, the proof of Lemma~\ref{lem:technical}.

\providecommand{\bysame}{\leavevmode\hbox to3em{\hrulefill}\thinspace}
\providecommand{\MR}{\relax\ifhmode\unskip\space\fi MR }
\providecommand{\MRhref}[2]{%
  \href{http://www.ams.org/mathscinet-getitem?mr=#1}{#2}
}
\providecommand{\href}[2]{#2}


\begin{thebibliography}{100}


\bibitem{AJOT} 
T. Abe, I. D. Jong, Y. Omae, and M. Takeuchi, 
{\em Annulus twist and diffeomorphic 4-manifolds}, 
Math. Proc. Cambridge Philos. Soc. 
{\bf 155} (2013), no. 2, 219--235.

\bibitem{AT} 
T. Abe and M. Tange,
{\em A construction of slice knots via annulus twists}, 
arXiv:1305.7492 (2013). 


\bibitem{Ak2}
S. Akbulut, 
{\em On $2$-dimensional homology classes of $4$-manifolds}, 
Math. Proc. Cambridge Philos. Soc. {\bf 82} (1977), no. 1, 99--106.

\bibitem{AkbulutBook}
S. Akbulut,
{\em $4$-manifolds}, draft of a book (2012), 
available at \\ 
{\tt http://www.math.msu.edu/\~{}akbulut/papers/akbulut.lec.eps}
 
\bibitem{bgl}
K. L. Baker, C. McA. Gordon, and J. Luecke, 
{\em Bridge number and integral Dehn surgery}, 
arXiv:1303.7018 (2013).


\bibitem{Cerf}
J. Cerf, 
{\em Sur les diffeomorphismes de la sphere de dimension trois $(\Gamma_{4}=0)$},
Lecture Notes in Mathematics, No. 53, Springer-Verlag, Berlin-New York (1968) xii+133 pp. 

\bibitem{snappy}
M. Culler, N. Dunfield, and J. R. Weeks, 
{\em SnapPy, a computer program for studying the geometry and topology of 3-manifolds}, 
{\tt http://snappy.computop.org}.



\bibitem{gordon2}
C. McA. Gordon, {\em Dehn surgery and satellite knots}, 
Trans. Amer. Math Soc. {\bf 275} (1983), no. 2, 687--708.

\bibitem{gordon}
C. McA. Gordon, 
{\em Boundary slopes of punctured tori in 3-manifolds}, Trans.
Amer. Math. Soc. {\bf 350} (1998),  no. 5, 1713--1790.

\bibitem{gordon-wu}
C. McA. Gordon and Y.-Q. Wu, 
{\em Toroidal Dehn fillings on hyperbolic 3-manifolds}, 
Memoirs AMS {\bf 194}, No. 909, (2008), 1--147. 

\bibitem{hikmot}
N. Hoffman, K. Ichihara, M. Kashiwagi, H. Masai, S. Oishi, and A. Takayasu,
{\em Verified computations for hyperbolic 3-manifolds}, 
arXiv:1310.3410 (2013), 
code available from: 

\texttt{http://www.oishi.info.waseda.ac.jp/\~{}takayasu/hikmot/}



\bibitem{Kirby}
R. Kirby, 
{\em Problems in low-dimensional topology}, 
AMS/IP Stud. Adv. Math. {\bf 2}(2), 
Geometric topology (Athens, GA, 1993), 35--473 (Amer. Math. Soc. 1997). 

\bibitem{kouno}
R.\ Kouno, 
{\em 3-manifolds with infinitely many knot surgery descriptions (in Japanese)}, 
Masters thesis, Nihon University (2002).


\bibitem{lickorish}
W. B. R. Lickorish, 
{\em A representation of orientable combinatorial 3-manifolds}, 
Ann. of Math {\bf 76} (1962), no. 3, 531--538. 
 

\bibitem{neumann}
W. Neumann and D. Zagier, {\em Volumes of hyperbolic three-manifolds}, 
Topology {\bf 24} (1985), no. 3, 307--332.

\bibitem{Osoinach}
J. Osoinach, 
{\em Manifolds obtained by surgery on an infinite number of knots in $S^3$}, 
Topology {\bf 45} (2006), no. 4, 725--733. 


\bibitem{Rolfsen}
D. Rolfsen, 
{\em Knots and Links}, 
Mathematics Lecture Series, No.~7. Publish or Perish, Inc., Berkeley, Calif., 1976.

\bibitem{teragaito}
M. Teragaito, 
{\em A Seifert fibered manifold with infinitely many knot-surgery descriptions},
Int. Math. Res. Not. {\bf 2007}, no. 9, Art. ID rnm028, 16 pp.


\bibitem{wallace}
A. Wallace, 
{\em Modifications and cobounding manifolds}, 
Can. J. Math. {\bf 12} (1960), 503--528.

\end{thebibliography}
\end{document}